\documentclass{amsart}

\usepackage{amssymb,bm,mathrsfs,enumerate,color,scalerel,array}
\usepackage{hyperref}

\usepackage{tikz}
\usetikzlibrary{svg.path}
\definecolor{orcid_color}{HTML}{A6CE39}

\hypersetup{pdfborder={0 0 0}} 
\DeclareRobustCommand{\orcidicon}{%
	\raisebox{.2mm}{\scalerel*{%
	\begin{tikzpicture}[xscale=1,yscale=-1,transform shape]
	\filldraw[color=orcid_color] svg {M256,128c0,70.7-57.3,128-128,128C57.3,256,0,198.7,0,128C0,57.3,57.3,0,128,0C198.7,0,256,57.3,256,128z};
	\filldraw[color=white] svg {M86.3,186.2H70.9V79.1h15.4v48.4V186.2z} svg {M108.9,79.1h41.6c39.6,0,57,28.3,57,53.6c0,27.5-21.5,53.6-56.8,53.6h-41.8V79.1z M124.3,172.4h24.5
		c34.9,0,42.9-26.5,42.9-39.7c0-21.5-13.7-39.7-43.7-39.7h-23.7V172.4z} svg {M88.7,56.8c0,5.5-4.5,10.1-10.1,10.1c-5.6,0-10.1-4.6-10.1-10.1c0-5.6,4.5-10.1,10.1-10.1
		C84.2,46.7,88.7,51.3,88.7,56.8z};
	\end{tikzpicture}}{|}}%
}
\newcommand{\orcid}[1]{\href{https://orcid.org/#1}{\orcidicon}}

\theoremstyle{plain}
\newtheorem{thm}{Theorem}[section]
\newtheorem{prop}[thm]{Proposition}
\newtheorem{lem}[thm]{Lemma}
\newtheorem{claim}{Claim}
\theoremstyle{definition}
\newtheorem{defn}[thm]{Definition}
\newtheorem{exam}[thm]{Example}
\newtheorem{assump}[thm]{Assumption}
\newtheorem{rem}[thm]{Remark}

\newcommand{\arxiv}[1]{arXiv:\href{http://arxiv.org/abs/#1}{#1}}

\newcommand{\hypergeometricseries}[5]{{}_{#1}F_{#2}\!\left(\left.\!\!\!\begin{array}{c} #3 \\ #4 \end{array}\!\!\right| #5 \right)}

\allowdisplaybreaks[4]

\begin{document}

\title[Tight relative $t$-designs on two shells]{Tight relative $t$-designs on two shells in hypercubes, and Hahn and Hermite polynomials}
\author{Eiichi Bannai}
\address{Faculty of Mathematics, Kyushu University (emeritus), Japan}
\email{bannai@math.kyushu-u.ac.jp}
\author{Etsuko Bannai}
\email{et-ban@rc4.so-net.ne.jp}
\author[Hajime Tanaka]{Hajime Tanaka\,\orcid{0000-0002-5958-0375}}
\address{\href{http://www.math.is.tohoku.ac.jp/}{Research Center for Pure and Applied Mathematics}, Graduate School of Information Sciences, Tohoku University, Sendai 980-8579, Japan}
\email{htanaka@tohoku.ac.jp}
\author{Yan Zhu}
\address{College of Science, University of Shanghai for Science and Technology, Shanghai 200093, China}
\email{zhuyan@usst.edu.cn}

\begin{abstract}
Relative $t$-designs in the $n$-dimensional hypercube $\mathcal{Q}_n$ are equivalent to weighted regular $t$-wise balanced designs, which generalize combinatorial $t$-$(n,k,\lambda)$ designs by allowing multiple block sizes as well as weights.
Partly motivated by the recent study on tight Euclidean $t$-designs on two concentric spheres, in this paper we discuss tight relative $t$-designs in $\mathcal{Q}_n$ supported on two shells.
We show under a mild condition that such a relative $t$-design induces the structure of a coherent configuration with two fibers.
Moreover, from this structure we deduce that a polynomial from the family of the Hahn hypergeometric orthogonal polynomials must have only integral simple zeros.
The Terwilliger algebra is the main tool to establish these results.
By explicitly evaluating the behavior of the zeros of the Hahn polynomials when they degenerate to the Hermite polynomials under an appropriate limit process, we prove a theorem which gives a partial evidence that the non-trivial tight relative $t$-designs in $\mathcal{Q}_n$ supported on two shells are rare for large $t$.
\end{abstract}

\maketitle

\hypersetup{pdfborder={0 0 1}} 

\section{Introduction}

This paper is a contribution to the study of \emph{relative} $t$-\emph{designs} in $Q$-polynomial association schemes.
In the Delsarte theory \cite{Delsarte1973PRRS}, the concept of $t$-\emph{designs} is introduced for arbitrary $Q$-polynomial association schemes.
For the Johnson scheme $J(n,k)$, the $t$-designs in the sense of Delsarte are shown to be the same thing as the combinatorial $t$-$(n,k,\lambda)$ designs.
There are similar interpretations of $t$-designs in some other important families of $Q$-polynomial association schemes \cite{Delsarte1973PRRS,Delsarte1976JCTA,Delsarte1978SIAM,Munemasa1986GC,Stanton1986GC}.
The concept of relative $t$-designs is also due to Delsarte \cite{Delsarte1977PRR}, and is a relaxation of that of $t$-designs.
Relative $t$-designs can again be interpreted in several cases, including $J(n,k)$.
For the $n$-dimensional hypercube $\mathcal{Q}_n$ (or the binary Hamming scheme $H(n,2)$) which will be our central focus in this paper, these are equivalent to the weighted regular $t$-\emph{wise balanced designs}, which generalize the combinatorial $t$-$(n,k,\lambda)$ designs by allowing multiple block sizes as well as weights.

The Delsarte theory has a counterpart for the unit sphere $\mathbb{S}^{n-1}$ in $\mathbb{R}^n$, established by Delsarte, Goethals, and Seidel \cite{DGS1977GD}.
The $t$-designs in $\mathbb{S}^{n-1}$ are commonly called the \emph{spherical} $t$-\emph{designs}, and are essentially the equally-weighted \emph{cubature formulas} of degree $t$ for the spherical integration, a concept studied extensively in numerical analysis.
Spherical $t$-designs were later generalized to \emph{Euclidean} $t$-\emph{designs} by Neumaier and Seidel \cite{NS1988IM} (cf.~\cite{DS1989LAA}).
Euclidean $t$-designs are in general supported on multiple concentric spheres in $\mathbb{R}^n$, and it follows that we may think of them as the natural counterpart of relative $t$-designs in $\mathbb{R}^n$.
This point of view was discussed in detail by Bannai and Bannai \cite{BB2012JAMC}.
See also \cite{BBTZ2017GC,BBZ2015PSIM}.
The success and the depth of the theory of Euclidean $t$-designs (cf.~\cite{SHK2019B}) has been one driving force for the recent research activity on relative $t$-designs in $Q$-polynomial association schemes; see, e.g., \cite{BB2012JAMC,BBB2014DM,BBST2015EJC,BBTZ2017GC,BBZ2015PSIM,BBZ2017DCC,BZ2019EJC,LBB2014GC,Xiang2012JCTA,YHG2015DM,ZBB2016DM}.

A relative $t$-design in a $Q$-polynomial association scheme $(X,\mathscr{R})$ is often defined as a certain \emph{weighted subset} of the vertex set $X$, i.e., a pair $(Y,\omega)$ of a subset $Y$ of $X$ and a function $\omega:Y\rightarrow (0,\infty)$.
We are given in advance a `base vertex' $x\in X$, and $(Y,\omega)$ gives a `degree-$t$ approximation' of the shells (or spheres or subconstituents) with respect to $x$ on which $Y$ is supported.
See Sections \ref{sec: association schemes} and \ref{sec: relative t-designs} for formal definitions.
Bannai and Bannai \cite{BB2012JAMC} proved a Fisher-type lower bound on $|Y|$, and we call $(Y,\omega)$ \emph{tight} if it attains this bound.
We may remark that $t$ must be even in this case.
In this paper, we continue the study (cf.~\cite{BBB2014DM,BBZ2017DCC,LBB2014GC,Xiang2012JCTA,YHG2015DM}) of tight relative $t$-designs in the hypercubes $\mathcal{Q}_n$, which are one of the most important families of $Q$-polynomial association schemes.
The Delsarte theory directly applies to the tight relative $t$-designs in $\mathcal{Q}_n$ supported on one shell, say, the $k^{\mathrm{th}}$ shell, as these are equivalent to the tight combinatorial $t$-$(n,k,\lambda)$ designs.
(We note that the $k^{\mathrm{th}}$ shell induces $J(n,k)$.)
Our aim is to extend this structure theory to those supported on two shells.
We may view the results of this paper roughly as counterparts to (part of) the results by Bannai and Bannai \cite{BB2010CM,BB2014MJCNT} on tight Euclidean $t$-designs on two concentric spheres.

Let $t=2e$ be even.
In Theorem \ref{counterpart of Wilson's theorem}, which is our first main result, we show under a mild condition that a tight relative $2e$-design in $\mathcal{Q}_n$ supported on two shells induces the structure of a coherent configuration with two fibers.
Moreover, from this structure we deduce that a certain polynomial of degree $e$, known as a \emph{Hahn polynomial}, must have only integral simple zeros.
We note that the case $e=1$ was handled previously by Bannai, Bannai, and Bannai \cite{BBB2014DM}.
The Hahn polynomials are a family of hypergeometric orthogonal polynomials in the Askey scheme \cite[Section 1.5]{KS1998R}, and that their zeros are integral provides quite a strong necessary condition on the existence of such relative $2e$-designs.
The corresponding necessary condition for the tight combinatorial $2e$-$(n,k,\lambda)$ designs from the Delsarte theory was used successfully by Bannai \cite{Bannai1977QJMO}; that is to say, he showed that, for each given integer $e\geqslant 5$, there exist only finitely many non-trivial tight $2e$-$(n,k,\lambda)$ designs, where $n$ and $k$ (and thus $\lambda$) vary.
See also \cite{DS-G2013JAC,Peterson1977OJM,Xiang2018JAC}.
We extend Bannai's method to prove our second main result, Theorem \ref{counterpart of Bannai's theorem}, which presents a version of his theorem for our case.

The sections other than Sections \ref{sec: main result} and \ref{sec: finiteness result} are organized as follows.
We collect the necessary background material in Sections \ref{sec: association schemes} and \ref{sec: relative t-designs}.
Section \ref{sec: relative t-designs} also includes a few general results on relative $t$-designs in $Q$-polynomial association schemes.
As in \cite{BBST2015EJC,Tanaka2009EJC}, our main tool in the analysis of relative $t$-designs is the \emph{Terwilliger algebra} \cite{Terwilliger1992JAC,Terwilliger1993JACa,Terwilliger1993JACb}, which is a non-commutative semisimple $\mathbb{C}$-algebra containing the adjacency algebra.
Section \ref{sec: Terwilliger algebra of hypercube} is devoted to detailed descriptions of the Terwilliger algebra of $\mathcal{Q}_n$.
It is well known (cf.~\cite{KLS2010B,KS1998R}) that the Hahn polynomials (${}_3F_2$) degenerate to the Hermite polynomials (${}_2F_0$) by an appropriate limit process, and a key in Bannai's method above was to evaluate precisely the behavior of the zeros of the Hahn polynomials in this process.
In Section \ref{sec: zeros}, we revisit this part of the method in a form suited to our purpose.
Our account will also be simpler than that in \cite{Bannai1977QJMO}.
In Appendix, we provide a proof of a number-theoretic result (Proposition \ref{variation of Schur's theorem}) which is a variation of a result of Schur \cite[Satz I]{Schur1929SPAW}.

\section{Coherent configurations and association schemes}\label{sec: association schemes}

We begin by recalling the concept of coherent configurations.

\begin{defn}
The pair $(X,\mathscr{R})$ of a finite set $X$ and a set $\mathscr{R}$ of non-empty subsets of $X^2$ is called a \emph{coherent configuration} on $X$ if it satisfies the following (C\ref{C1})--(C\ref{C4}):
\begin{enumerate}[({C}1)]
\item\label{C1} $\mathscr{R}$ is a partition of $X^2$.
\item There is a subset $\mathscr{R}_0$ of $\mathscr{R}$ such that
\begin{equation*}
	\bigsqcup_{R\in\mathscr{R}_0} \!\! R=\{(x,x):x\in X\}.
\end{equation*}
\item\label{C3} $\mathscr{R}$ is invariant under the transposition $\tau:(x,y)\mapsto (y,x)$ $((x,y)\in X^2)$, i.e., $R^{\tau}\in \mathscr{R}$ for all $R\in\mathscr{R}$.
\item\label{C4} For all $R,S,T\in\mathscr{R}$ and $(x,y)\in T$, the number
\begin{equation*}
	p_{R,S}^T:=\bigl|\{z\in X:(x,z)\in R,\,(z,y)\in S\}\bigr|
\end{equation*}
is independent of the choice of $(x,y)\in T$.
\end{enumerate}
Moreover, a coherent configuration $(X,\mathscr{R})$ on $X$ is called \emph{homogeneous} if $|\mathscr{R}_0|=1$, and an \emph{association scheme} if $R^{\tau}=R$ for all $R\in\mathscr{R}$.
\end{defn}

\begin{rem}\label{Schurian}
Suppose that a finite group $\mathfrak{G}$ acts on $X$, and let $\mathscr{R}$ be the set of the orbitals of $\mathfrak{G}$, that is to say, the orbits of $\mathfrak{G}$ in its natural action on $X^2$.
Then $(X,\mathscr{R})$ is a coherent configuration.
Moreover, $(X,\mathscr{R})$ is homogeneous (resp.~an association scheme) if and only if the action of $\mathfrak{G}$ on $X$ is transitive (resp.~generously transitive, i.e., for any $x,y\in X$ we have $(x^g,y^g)=(y,x)$ for some $g\in\mathfrak{G}$).
\end{rem}

Let $(X,\mathscr{R})$ be a coherent configuration as above.
For every $R\in\mathscr{R}_0$, let $\Phi_R$ be the subset of $X$ such that $R=\{(x,x):x\in \Phi_R\}$.
Then we have
\begin{equation*}
	\bigsqcup_{R\in\mathscr{R}_0}\!\! \Phi_R=X.
\end{equation*}
We call the $\Phi_R$ $(R\in\mathscr{R}_0)$ the \emph{fibers} of $(X,\mathscr{R})$.
By setting in (C\ref{C4}) either $R\in\mathscr{R}_0$ and $S=T$, or $S\in\mathscr{R}_0$ and $R=T$, it follows that for every $T\in\mathscr{R}$, we have $T\subset \Phi_R\times \Phi_S$ for some $R,S\in\mathscr{R}_0$.
In particular, $(X,\mathscr{R})$ is homogeneous whenever it is an association scheme.
Let
\begin{equation*}
	\gamma_{R,S}=\bigl|\{T\in\mathscr{R}:T\subset \Phi_R\times \Phi_S\}\bigr| \qquad (R,S\in\mathscr{R}_0).
\end{equation*}
The matrix
\begin{equation*}
	[\gamma_{R,S}]_{R,S\in\mathscr{R}_0},
\end{equation*}
which is symmetric by (C\ref{C3}), is called the \emph{type} of $(X,\mathscr{R})$.

Let $M_X(\mathbb{C})$ be the $\mathbb{C}$-algebra of all complex matrices with rows and columns indexed by $X$, and let $V=\mathbb{C}^X$ be the $\mathbb{C}$-vector space of complex column vectors with coordinates indexed by $X$.
We endow $V$ with the Hermitian inner product
\begin{equation*}
	\langle u,v\rangle=v^{\dagger}u \qquad (u,v\in V),
\end{equation*}
where $^{\dagger}$ denotes adjoint.
For every $R\in\mathscr{R}$, let $A_R\in M_X(\mathbb{C})$ be the adjacency matrix of the graph $(X,R)$ (directed, in general), i.e.,
\begin{equation*}
	(A_R)_{x,y}=\begin{cases} 1 & \text{if} \ (x,y)\in R, \\ 0 & \text{otherwise}, \end{cases} \qquad (x,y\in X).
\end{equation*}
Then (C\ref{C1})--(C\ref{C4}) above are rephrased as follows:
\begin{enumerate}[({A}1)]
\item\label{A1} $\displaystyle\sum_{R\in\mathscr{R}}A_R=J$ (the all-ones matrix).
\item\label{A2} $\displaystyle\sum_{R\in\mathscr{R}_0} \!\! A_R=I$ (the identity matrix).
\item\label{A3} $(A_R)^{\dagger}\in\{A_S:S\in \mathscr{R}\}$ \ $(R\in\mathscr{R})$.
\item\label{A4} $\displaystyle A_RA_S=\sum_{T\in\mathscr{R}}p_{R,S}^T A_T$ \ $(R,S\in\mathscr{R})$.
\end{enumerate}
Let
\begin{equation*}
	\bm{A}=\operatorname{span}\{A_R:R\in \mathscr{R}\}.
\end{equation*}
Then from (A\ref{A2}) and (A\ref{A4}) it follows that $\bm{A}$ is a subalgebra of $M_X(\mathbb{C})$, called the \emph{adjacency algebra} of $(X,\mathscr{R})$.
We note that $\bm{A}$ is semisimple as it is closed under $^{\dagger}$ by virtue of (A\ref{A3}).
By (A\ref{A1}), $\bm{A}$ is also closed under entrywise (or \emph{Hadamard} or \emph{Schur}) multiplication, which we denote by $\circ$.
The $A_R$ are the (central) primitive idempotents of $\bm{A}$ with respect to $\circ$, i.e.,
\begin{equation*}
	A_R\circ A_S=\delta_{R,S}A_R, \qquad \sum_{R\in\mathscr{R}}A_R=J.
\end{equation*}

\begin{rem}\label{centralizer algebra}
If $(X,\mathscr{R})$ arises from a group action as in Remark \ref{Schurian}, then $\bm{A}$ coincides with the centralizer algebra (or Hecke algebra or commutant) for the corresponding permutation representation $g\mapsto P_g$ $(g\in\mathfrak{G})$ on $V$, i.e.,
\begin{equation*}
	\bm{A}=\{B\in M_X(\mathbb{C}): BP_g=P_gB \ (g\in\mathfrak{G})\}.
\end{equation*}
\end{rem}

A subalgebra of $M_X(\mathbb{C})$ is called a \emph{coherent algebra} if it contains $J$, and is closed under $\circ$ and $^{\dagger}$.
We note that the coherent algebras are precisely the adjacency algebras of coherent configurations.
It is clear that the intersection of coherent algebras in $M_X(\mathbb{C})$ is again a coherent algebra.
In particular, for any subset $\bm{S}$ of $M_X(\mathbb{C})$, we can speak of \emph{the} smallest coherent algebra containing $\bm{S}$, which we call the \emph{coherent closure} of $\bm{S}$.

From now on, we assume that $(X,\mathscr{R})$ is an association scheme.
As is the case for many examples of association schemes, we write
\begin{equation*}
	\mathscr{R}=\{R_0,R_1,\dots,R_n\}, \quad \text{where} \ \ \mathscr{R}_0=\{R_0\},
\end{equation*}
and say that $(X,\mathscr{R})$ has $n$ \emph{classes}.
We will then abbreviate $p_{i,j}^k=p_{R_i,R_j}^{R_k}$, $A_i=A_{R_i}$, and so on.
The adjacency algebra $\bm{A}$ is commutative in this case, and hence it has a basis $E_0,E_1,\dots,E_n$ consisting of the (central) primitive idempotents, i.e.,
\begin{equation*}
	E_iE_j=\delta_{i,j}E_i, \qquad \sum_{i=0}^nE_i=I.
\end{equation*}
Put differently, $E_0V,E_1V,\dots,E_nV$ are the maximal common eigenspaces (or homogeneous components or isotypic components) of $\bm{A}$, and the $E_i$ are the corresponding orthogonal projections.
Since the $A_i$ are real symmetric matrices, so are the $E_i$.
Note that the matrix $|X|^{-1}J\in\bm{A}$ is an idempotent with rank one, and thus primitive.
We will always set
\begin{equation*}
	E_0=\frac{1}{|X|}J.
\end{equation*}
For convenience, we let
\begin{equation*}
	A_i=E_i:=O \ \text{(the zero matrix)} \quad \text{if} \ i<0 \ \text{or} \ i>n.
\end{equation*}

Though our focus in this paper will be on $Q$-polynomial association schemes, we first recall the $P$-polynomial property for completeness.
We say that the association scheme $(X,\mathscr{R})$ is $P$-\emph{polynomial} (or \emph{metric}) with respect to the ordering $A_0,A_1,\dots,A_n$ if there are non-negative integers $a_i,b_i,c_i$ $(0\leqslant i\leqslant n)$ such that $b_n=c_0=0$, $b_{i-1}c_i\ne 0$ $(1\leqslant i\leqslant n)$, and
\begin{equation*}
	A_1A_i=b_{i-1}A_{i-1}+a_iA_i+c_{i+1}A_{i+1} \qquad (0\leqslant i\leqslant n),
\end{equation*}
where $b_{-1}$ and $c_{n+1}$ are indeterminates.
In this case, $A_1$ recursively generates $\bm{A}$, and hence has $n+1$ distinct eigenvalues $\theta_0,\theta_1,\dots,\theta_n\in\mathbb{R}$, where we write
\begin{equation}\label{eigenvalues}
	A_1=\sum_{i=0}^n \theta_iE_i.
\end{equation}
We note that $(X,\mathscr{R})$ is $P$-polynomial as above precisely when the graph $(X,R_1)$ is a \emph{distance-regular graph} and $(X,R_i)$ is the distance-$i$ graph of $(X,R_1)$ $(0\leqslant i\leqslant n)$.
See, e.g., \cite{BI1984B,BCN1989B,DKT2016EJC,Godsil1993B} for more information on distance-regular graphs.

We say that $(X,\mathscr{R})$ is $Q$-\emph{polynomial} (or \emph{cometric}) with respect to the ordering $E_0,E_1,\dots,E_n$ if there are real scalars $a_i^*,b_i^*,c_i^*$ $(0\leqslant i\leqslant n)$ such that $b_n^*=c_0^*=0$, $b_{i-1}^*c_i^*\ne 0$ $(1\leqslant i\leqslant n)$, and
\begin{equation}\label{*3-term recurrence}
	E_1\circ E_i=\frac{1}{|X|}(b_{i-1}^*E_{i-1}+a_i^*E_i+c_{i+1}^*E_{i+1}) \qquad (0\leqslant i\leqslant n),
\end{equation}
where $b_{-1}^*$ and $c_{n+1}^*$ are indeterminates.
In this case, $|X|E_1$ recursively generates $\bm{A}$ with respect to $\circ$, and hence has $n+1$ distinct entries $\theta_0^*,\theta_1^*,\dots,\theta_n^*\in\mathbb{R}$, where we write
\begin{equation}\label{dual eigenvalues}
	|X|E_1=\sum_{i=0}^n\theta_i^*A_i.
\end{equation}
We call the $\theta_i^*$ the \emph{dual eigenvalues} of $|X|E_1$.
We may remark that $E_1\circ E_i$, being a principal submatrix of $E_1\otimes E_i$, is positive semidefinite, so that the scalars $a_i^*,b_i^*$, and $c_i^*$ are non-negative (the so-called \emph{Krein condition}).
The $Q$-polynomial association schemes are an important subject in their own right, and we refer the reader to \cite{GVW2019pre,Kodalen2019D} and the references therein for recent activity.

Below we give two fundamental examples of $P$- and $Q$-polynomial association schemes, both of which come from transitive group actions.
See \cite{BI1984B,BCN1989B,Delsarte1973PRRS} for the details.

\begin{exam}\label{Johnson}
Let $v$ and $k$ be positive integers with $v>k$, and let $X$ be the set of $k$-subsets of $\{1,2,\dots,v\}$.
Set $n=\min\{k,v-k\}$.
For $x,y\in X$ and $0\leqslant i\leqslant n$, we let $(x,y)\in R_i$ if $|x\cap y|=k-i$.
The $R_i$ are the orbitals of the symmetric group $\mathfrak{S}_v$ acting on $X$.
We call $(X,\mathscr{R})$ a \emph{Johnson scheme} and denote it by $J(v,k)$.
The eigenvalues of $A_1$ are given in decreasing order by
\begin{equation*}
	\theta_i=(k-i)(v-k-i)-i \qquad (0\leqslant i\leqslant n),
\end{equation*}
and $J(v,k)$ is $Q$-polynomial with respect to the corresponding ordering of the $E_i$ (cf.~\eqref{eigenvalues}).
\end{exam}

\begin{exam}\label{Hamming}
Let $q\geqslant 2$ be an integer and let $X=\{0,1,\dots,q-1\}^n$.
For $x,y\in X$ and $0\leqslant i\leqslant n$, we let $(x,y)\in R_i$ if $x$ and $y$ differ in exactly $i$ coordinate positions.
The $R_i$ are the orbitals of the wreath product $\mathfrak{S}_q\wr\mathfrak{S}_n$ of the symmetric groups $\mathfrak{S}_q$ and $\mathfrak{S}_n$ acting on $X$.
We call $(X,\mathscr{R})$ a \emph{Hamming scheme} and denote it by $H(n,q)$.
The eigenvalues of $A_1$ are given in decreasing order by
\begin{equation*}
	\theta_i=n(q-1)-qi \qquad (0\leqslant i\leqslant n),
\end{equation*}
and $H(n,q)$ is $Q$-polynomial with respect to the corresponding ordering of the $E_i$ (cf.~\eqref{eigenvalues}).
The Hamming scheme $H(n,2)$ is also known as the $n$-\emph{cube} (or $n$-\emph{dimensional hypercube}) and is denoted by $\mathcal{Q}_n$.
\end{exam}

\begin{assump}\label{assume Q-polynomial}
For the rest of this section and in Section \ref{sec: relative t-designs}, we assume that $(X,\mathscr{R})$ is an association scheme and is $Q$-polynomial with respect to the ordering $E_0,E_1,\dots,E_n$ of the primitive idempotents.
\end{assump}

In general, for any positive semidefinite Hermitian matrices $B,C\in M_X(\mathbb{C})$, we have (cf.~\cite{Tanaka2009LAAa})
\begin{equation*}
	(B\circ C)V=\operatorname{span}(BV\circ CV),
\end{equation*}
where
\begin{equation*}
	BV\circ CV=\{u\circ v: u\in BV,\, v\in CV\}.
\end{equation*}
Hence it follows from \eqref{*3-term recurrence} that
\begin{equation}\label{Hadamard with E1}
	\operatorname{span} (E_1V\circ E_iV)=\begin{cases} E_{i-1}V+E_iV+E_{i+1}V & \text{if} \  a_i^*\ne 0, \\ E_{i-1}V+E_{i+1}V & \text{if} \  a_i^*=0, \end{cases} \quad (0\leqslant i\leqslant n),
\end{equation}
from which it follows that
\begin{align}
	\sum_{i=0}^h\sum_{j=0}^k \operatorname{span}(E_iV\circ E_jV) &= \sum_{i=0}^h\sum_{j=0}^k \operatorname{span}(\underbrace{E_1V\circ\cdots\circ E_1V}_{i\,\text{times}}\circ E_jV) \label{Hadamard of flags} \\
	&= \sum_{i=0}^{h+k} E_iV \notag
\end{align}
for $0\leqslant h,k\leqslant n$.
See also \cite[Section 2.8]{BI1984B}.

We now fix a `base vertex' $x\in X$.
Let
\begin{equation*}
	X_i=\{y\in X:(x,y)\in R_i\} \qquad (0\leqslant i\leqslant n).
\end{equation*}
We call the $X_i$ the \emph{shells} (or \emph{spheres} or \emph{subconstituents}) of $(X,\mathscr{R})$ with respect to $x$.
For every $i$ $(0\leqslant i\leqslant n)$, define the diagonal matrix $E_i^*=E_i^*(x)\in M_X(\mathbb{C})$ by
\begin{equation*}
	(E_i^*)_{y,y}=\begin{cases} 1 & \text{if} \ y\in X_i, \\ 0 & \text{otherwise}, \end{cases} \qquad (y\in X).
\end{equation*}
Then we have
\begin{equation*}
	E_i^*E_j^*=\delta_{i,j}E_i^*, \qquad \sum_{i=0}^nE_i^*=I.
\end{equation*}
We call the $E_i^*$ the \emph{dual idempotents} of $(X,\mathscr{R})$ with respect to $x$.
The subspace
\begin{equation*}
	\bm{A}^*=\bm{A}^*(x)=\operatorname{span}\{E_0^*,E_1^*,\dots,E_n^*\}
\end{equation*}
is then a subalgebra of $M_X(\mathbb{C})$, which we call the \emph{dual adjacency algebra} of $(X,\mathscr{R})$ with respect to $x$.
The \emph{Terwilliger algebra} (or \emph{subconstituent algebra}) of $(X,\mathscr{R})$ with respect to $x$ is the subalgebra $\bm{T}=\bm{T}(x)$ of $M_X(\mathbb{C})$ generated by $\bm{A}$ and $\bm{A}^*$ \cite{Terwilliger1992JAC,Terwilliger1993JACa,Terwilliger1993JACb}.
We note that $\bm{T}$ is semisimple as it is closed under $^{\dagger}$.

\begin{rem}
If $(X,\mathscr{R})$ arises from a group action as in Remark \ref{Schurian}, which we recall is generously transitive in this case, then $\bm{T}$ is a subalgebra of the centralizer algebra for the action of the stabilizer $\mathfrak{G}_x$ of $x$ in $\mathfrak{G}$.
The two algebras are known to be equal, e.g., for $J(v,k)$ and $H(n,q)$; see \cite{GST2006JCTA,TFIL2019EJC}.
\end{rem}

For every subset $Y$ of $X$, let $\hat{Y}\in V$ be the characteristic vector of $Y$, i.e.,
\begin{equation*}
	(\hat{Y})_y=\begin{cases} 1 & \text{if} \ y\in Y, \\ 0 & \text{otherwise}, \end{cases} \qquad (y\in X).
\end{equation*}
In particular, $\hat{X}$ denotes the all-ones vector in $V$.
We will simply write $\hat{x}$ for the characteristic vector of the singleton $\{x\}$.
With this notation established, we have
\begin{equation*}
	\hat{X}_i=E_i^*\hat{X}=A_i\hat{x} \qquad (0\leqslant i\leqslant n),
\end{equation*}
from which it follows that
\begin{equation}\label{primary T-module}
	\bm{T}\hat{x}=\operatorname{span}\{\hat{X}_i: 0\leqslant i\leqslant n\}=\operatorname{span}\{E_i\hat{x}: 0\leqslant i\leqslant n \}.
\end{equation}
The $\bm{T}$-module $\bm{T}\hat{x}$ is easily seen to be irreducible with dimension $n+1$ (cf.~\cite[Lemma 3.6]{Terwilliger1992JAC}), and is called the \emph{primary} $\bm{T}$-module.

We define the \emph{dual adjacency matrix} $A_1^*=A_1^*(x)\in M_X(\mathbb{C})$ by (cf.~\eqref{dual eigenvalues})
\begin{equation}\label{dual adjacency matrix}
	A_1^*=|X| \operatorname{diag} E_1\hat{x}=\sum_{i=0}^n\theta_i^*E_i^*.
\end{equation}
Since the $\theta_i^*$ are mutually distinct, $A_1^*$ generates $\bm{A}^*$.
Moreover, since
\begin{equation*}
	A_1^*v=|X|(E_1\hat{x})\circ v \qquad (v\in V),
\end{equation*}
it follows from \eqref{Hadamard with E1} that
\begin{equation}\label{tridiagonal action}
	E_iA_1^*E_j=O \quad \text{if} \ |i-j|>1 \quad (0\leqslant i,j\leqslant n).
\end{equation}

Let $W$ be an irreducible $\bm{T}$-module.
We define the \emph{dual support} $W_s^*$, the \emph{dual endpoint} $r^*(W)$, and the \emph{dual diameter} $d^*(W)$ of $W$ by
\begin{equation*}
	W_s^*=\{i:E_iW\ne 0\}, \qquad r^*(W)=\min W_s^*, \qquad d^*(W)=|W_s^*|-1,
\end{equation*}
respectively.
We call $W$ \emph{dual thin} if $\dim E_iW\leqslant 1$ $(0\leqslant i\leqslant n)$.
We note that the primary $\bm{T}$-module $\bm{T}\hat{x}$ is dual thin, and that it is a unique irreducible $\bm{T}$-module up to isomorphism which has dual endpoint zero or dual diameter $n$.
The following lemma is an easy consequence of \eqref{tridiagonal action}:

\begin{lem}[{\cite[Lemma 3.12]{Terwilliger1992JAC}}]\label{Q-property}
With reference to Assumption \ref{assume Q-polynomial}, write $A_1^*=A_1^*(x)$, $\bm{A}^*=\bm{A}^*(x)$, $\bm{T}=\bm{T}(x)$.
Let $W$ be an irreducible $\bm{T}$-module and set $r^*=r^*(W)$, $d^*=d^*(W)$.
Then the following hold:
\begin{enumerate}
\item $A_1^*E_iW\subset E_{i-1}W+E_iW+E_{i+1}W$ $(0\leqslant i\leqslant n)$.
\item $W_s^*=\{r^*,r^*+1,\dots,r^*+d^*\}$.
\item $E_iA_1^*E_jW\ne 0$ if $|i-j|=1$ $(r^*\leqslant i,j\leqslant r^*+d^*)$.
\item Suppose that  $W$ is dual thin.
Then
\begin{equation*}
	\sum_{h=0}^i E_{r^*+h}W = \sum_{h=0}^i (A_1^*)^h E_{r^*}W \qquad (0\leqslant i\leqslant d^*).
\end{equation*}
In particular, $W=\bm{A}^*E_{r^*}W$.
\end{enumerate}
\end{lem}

\section{Relative \texorpdfstring{$t$}{t}-designs in \texorpdfstring{$Q$}{Q}-polynomial association schemes}\label{sec: relative t-designs}

In this section, we develop some general theory on relative $t$-designs in $Q$-polynomial association schemes.

Recall Assumption \ref{assume Q-polynomial}.
Throughout this section, we fix a base vertex $x\in X$, and write $E_i^*=E_i^*(x)$ $(0\leqslant i\leqslant n)$, $A_1^*=A_1^*(x)$, $\bm{A}^*=\bm{A}^*(x)$, and $\bm{T}=\bm{T}(x)$.
In Introduction, we meant by a weighted subset of $X$ a pair $(Y,\omega)$ of a subset $Y$ of $X$ and a function $\omega:Y\rightarrow (0,\infty)$.
For convenience, however, we extend the domain of $\omega$ to $X$ by setting $\omega(y)=0$ for every $y\in X\backslash Y$.
We will also naturally identify $V$ with the set of complex functions on $X$, so that $\omega\in V$ and $Y=\operatorname{supp}\omega$.
In our discussions on relative $t$-designs, we will often consider the set
\begin{equation}\label{L}
	L=L_Y=\{\ell:Y\cap X_{\ell}\ne\emptyset\},
\end{equation}
and say that $(Y,\omega)$ is \emph{supported on} $\bigsqcup_{\ell\in L}X_{\ell}$.

For comparison, we begin with the algebraic definition of $t$-designs in $(X,\mathscr{R})$ due to Delsarte \cite{Delsarte1973PRRS,Delsarte1976JCTA}.
\begin{defn}
A weighted subset $(Y,\omega)$ of $X$ is called a $t$-\emph{design} in $(X,\mathscr{R}$) if $E_i\omega=0$ for $1\leqslant i\leqslant t$.
\end{defn}
\noindent
Delsarte \cite{Delsarte1977PRR} generalized this concept as follows:
\begin{defn}
A weighted subset $(Y,\omega)$ of $X$ is called a \emph{relative} $t$-\emph{design} in $(X,\mathscr{R}$) (with respect to $x$) if $E_i\omega\in\operatorname{span}\{E_i\hat{x}\}$ for $1\leqslant i\leqslant t$.
\end{defn}

\begin{rem}
Delsarte introduced the concept of $t$-designs for subsets $Y$ of $X$ in \cite{Delsarte1973PRRS}, i.e., when $\omega=\hat{Y}$, whereas in \cite{Delsarte1976JCTA,Delsarte1977PRR} he mostly considered general (i.e., not necessarily non-negative) non-zero vectors $\omega\in V$ in the discussions on $t$-designs and relative $t$-designs.
Some facts/results below, such as Examples \ref{designs in Johnson schemes} and \ref{designs in Hamming schemes}, Proposition \ref{akin to Euclidean case}, and Theorem \ref{Assmus-Mattson}, are still valid for general $\omega\in V$, but the Fisher-type lower bound on $|Y|=|\operatorname{supp}\omega|$ (cf.~Theorem \ref{Fisher}) makes sense only when $\omega$ is non-negative.
\end{rem}

For the Johnson and Hamming schemes, Delsarte \cite{Delsarte1973PRRS,Delsarte1976JCTA,Delsarte1977PRR} showed that these algebraic concepts indeed have geometric interpretations:

\begin{exam}\label{designs in Johnson schemes}
Let $(X,\mathscr{R})$ be the Johnson scheme $J(v,k)$ from Example \ref{Johnson}.
Then $(Y,\omega)$ is a $t$-design if and only if, for every $t$-subset $z$ of $\{1,2,\dots,v\}$, the sum $\lambda_z$ of the values $\omega(y)$ over those $y\in Y$ such that $z\subset y$, is a constant independent of $z$.
On the other hand, $(Y,\omega)$ is a relative $t$-design if and only if the above $\lambda_z$ depends only on $|x\cap z|$.
We note that $(Y,\hat{Y})$ is a $t$-design if and only if $Y$ is a $t$-$(v,k,\lambda)$ design (cf.~\cite[Chapter II.4]{CD2007H}) for some $\lambda$.
\end{exam}

\begin{exam}\label{designs in Hamming schemes}
Let $(X,\mathscr{R})$ be the Hamming scheme $H(n,q)$ from Example \ref{Hamming}.
Then $(Y,\omega)$ is a $t$-design if and only if, for every $t$-subset $\mathscr{T}$ of $\{1,2,\dots,n\}$ and every function $f:\mathscr{T}\rightarrow\{0,1,\dots,q-1\}$, the sum $\lambda_{\mathscr{T},f}$ of the values $\omega(y)$ over those $y=(y_1,y_2,\dots,y_n)\in Y$ such that $y_i=f(i)$ $(i\in\mathscr{T})$, is a constant independent of the pair $(\mathscr{T},f)$.
On the other hand, $(Y,\omega)$ is a relative $t$-design if and only if the above $\lambda_{\mathscr{T},f}$ depends only on $|\{i\in\mathscr{T}:x_i=f(i)\}|$, where $x=(x_1,x_2,\dots,x_n)$.
We note that $(Y,\hat{Y})$ is a $t$-design if and only if the transpose of the $|Y|\times n$ matrix formed by arranging the elements of $Y$ (in any order) is an orthogonal array $\operatorname{OA}(|Y|,n,q,t)$ (cf.~\cite[Chapter III.6]{CD2007H}).
For the case $q=2$, i.e., for $\mathcal{Q}_n$, if we choose the base vertex as $x=(0,0,\dots,0)$, then $(Y,\hat{Y})$ is a relative $t$-design if and only if $Y$ is a regular $t$-wise balanced design of type $t$-$(n,L,\lambda)$ (cf.~\cite[Section 4.4]{SHK2019B}) for some $\lambda$, where $L$ is from \eqref{L}, and where we identify the elements of $X=\{0,1\}^n$ with their supports.
\end{exam}

\noindent
Similar results hold for some other important families of $P$- and $Q$-polynomial association schemes; see, e.g., \cite{Delsarte1976JCTA,Delsarte1977PRR,Delsarte1978SIAM,Munemasa1986GC,Stanton1986GC}.

\begin{prop}[{cf.~\cite[Theorem 4.5]{BB2012JAMC}}]\label{akin to Euclidean case}
With reference to Assumption \ref{assume Q-polynomial}, let $(Y,\omega)$ be a weighted subset supported on $\bigsqcup_{\ell\in L}X_{\ell}$.
Then we have
\begin{equation}\label{projection onto primary T-module}
	\omega|_{\bm{T}\hat{x}}=\sum_{\ell\in L}\frac{\langle \omega,\hat{X}_{\ell}\rangle}{|X_{\ell}|}\hat{X}_{\ell},
\end{equation}
where $\omega|_{\bm{T}\hat{x}}$ denotes the orthogonal projection of $\omega$ on the primary $\bm{T}$-module $\bm{T}\hat{x}$.
Moreover, $(Y,\omega)$ is a relative $t$-design if and only if
\begin{equation*}
	\langle \omega,v\rangle=\langle \omega|_{\bm{T}\hat{x}},v\rangle =\sum_{\ell\in L}\frac{\langle \omega,\hat{X}_{\ell}\rangle}{|X_{\ell}|}\langle \hat{X}_{\ell},v\rangle
\end{equation*}
for every $v\in\sum_{i=0}^tE_iV$.
\end{prop}

\begin{proof}
Recall \eqref{primary T-module}.
The first part follows since the $\hat{X}_i$ form an orthogonal basis of $\bm{T}\hat{x}$ with $\|\hat{X}_i\|^2=|X_i|$.
The second part is also immediate from
\begin{equation*}
	E_i\omega\in\operatorname{span}\{E_i\hat{x}\} \ \Longleftrightarrow \ E_i\omega\in \bm{T}\hat{x} \ \Longleftrightarrow \ E_i\omega|_{\bm{T}\hat{x}}=E_i\omega. \qedhere
\end{equation*}
\end{proof}

\begin{rem}\label{remove trivial part}
It is clear that $(X_{\ell},\hat{X}_{\ell})$ is a relative $n$-design for every $0\leqslant \ell\leqslant n$.
Hence, if $(Y,\omega)$ is a relative $t$-design such that $X_{\ell}\subset Y$ for some $\ell$, and if $\omega$ is constant on $X_{\ell}$, then the weighted subset $(Y\backslash X_{\ell},(I-E_{\ell}^*)\omega)$ obtained by discarding $X_{\ell}$ from $Y$ is again a relative $t$-design.
This observation is particularly important when applying Theorem \ref{Assmus-Mattson} below; for example, we can always assume that $0\not\in L$.
\end{rem}

The following is a slight generalization of Delsarte's Assmus--Mattson theorem for $Q$-polynomial association schemes \cite[Theorem 8.4]{Delsarte1977PRR}, and can also be viewed as a variation of \cite[Theorem 3.3]{BBZ2017DCC}, which in turn generalizes \cite[Proposition 1]{Kageyama1991AC}.
See also \cite{BZ2019EJC}.
The proof is in fact identical to that of \cite[Theorem 4.3]{Tanaka2009EJC}, but we include it below because of the potential importance of the result.

\begin{thm}\label{Assmus-Mattson}
With reference to Assumption \ref{assume Q-polynomial}, let $(Y,\omega)$ be a relative $t$-design supported on $\bigsqcup_{\ell\in L}X_{\ell}$.
Then $(Y\cap X_{\ell},E_{\ell}^*\omega)$ is a relative $(t-|L|+1)$-design for every $\ell\in L$.
\end{thm}

\begin{proof}
Let $U=(\bm{T}\hat{x})^{\perp}$ be the orthogonal complement of $\bm{T}\hat{x}$ in $V$, which we recall is the sum of all the non-primary irreducible $\bm{T}$-modules in $V$.
On the one hand, we have
\begin{equation*}
	\omega |_U\in \sum_{\ell\in L}E_{\ell}^*U.
\end{equation*}
Since $A_1^*$ generates $\bm{A}^*$ and has at most $|L|$ distinct eigenvalues on this subspace (cf.~\eqref{dual adjacency matrix}), it follows that
\begin{equation}\label{spanning set}
	\bm{A}^*\omega|_U = \operatorname{span}\bigl\{\omega|_U,A_1^*\omega|_U,\dots,(A_1^*)^{|L|-1}\omega|_U\bigr\}.
\end{equation}
On the other hand, since $E_0U=0$, that $(Y,\omega)$ is a relative $t$-design is rephrased as
\begin{equation*}
	\omega |_U\in \sum_{i=t+1}^n \!\! E_iU.
\end{equation*}
Hence it follows from \eqref{tridiagonal action} and \eqref{spanning set} that
\begin{equation*}
	\bm{A}^*\omega |_U\subset \sum_{k=0}^{|L|-1} (A_1^*)^k \! \sum_{i=t+1}^n \!\! E_iU \subset \sum_{i=t-|L|+2}^n \!\!\! E_iU.
\end{equation*}
In particular, for every $\ell\in L$ we have
\begin{equation*}
	E_{\ell}^*\omega |_U \in \sum_{i=t-|L|+2}^n \!\!\! E_iU.
\end{equation*}
In other words, $(Y\cap X_{\ell},E_{\ell}^*\omega)$ is a relative $(t-|L|+1)$-design, as desired.
\end{proof}

Bannai and Bannai \cite[Theorem 4.8]{BB2012JAMC} established the following Fisher-type lower bound on the size of a relative $t$-design with $t$ even:

\begin{thm}\label{Fisher}
With reference to Assumption \ref{assume Q-polynomial}, let $(Y,\omega)$ be a relative $2e$-design $(e\in\mathbb{N})$ supported on $\bigsqcup_{\ell\in L}X_{\ell}$.
Then
\begin{equation*}
	|Y|\geqslant \dim \left(\sum_{\ell\in L}E_{\ell}^*\right)\!\!\left(\sum_{i=0}^eE_iV\right).
\end{equation*}
\end{thm}

\begin{defn}
A relative $2e$-design $(Y,\omega)$ is called \emph{tight} if equality holds above.
\end{defn}

Recall from Example \ref{designs in Hamming schemes} that the relative $t$-designs in the hypercubes are equivalent to the weighted regular $t$-wise balanced designs.

\begin{exam}
Let $(X,\mathscr{R})$ be the $n$-cube $\mathcal{Q}_n$ from Example \ref{Hamming}.
Xiang \cite{Xiang2012JCTA} showed that if $e\leqslant\ell\leqslant n-e$ for every $\ell\in L$, then
\begin{equation}\label{explicit lower bound}
	\dim \left(\sum_{\ell\in L}E_{\ell}^*\right)\!\!\left(\sum_{i=0}^eE_iV\right)=\sum_{i=0}^{\min\{|L|-1,e\}} \!\! \binom{n}{e-i}.
\end{equation}
We may remark that (cf.~\cite[Theorem 9.2.1]{BCN1989B})
\begin{equation}\label{multiplicities}
	\dim E_iV=\binom{n}{i} \qquad (0\leqslant i\leqslant n).
\end{equation}
See also \cite{LBB2014GC} and \cite[Theorem 2.7, Example 2.9]{BBST2015EJC}.
\end{exam}

\begin{exam}\label{tight relative 2-designs}
Consider a symmetric $2$-$(n+1,k,\lambda)$ design (cf.~\cite[Chapter II.6]{CD2007H}).
Observe that removing a point yields a tight relative $2$-design in $\mathcal{Q}_n$ with $L=\{k-1,k\}$.
Likewise, taking the complement of every block which contains a given point followed by removing that point gives rise to a tight relative $2$-design in $\mathcal{Q}_n$ with $L=\{k,n+1-k\}$.
The complement of this is yet another example\footnote{It seems that this construction is missing in \cite[Theorem 8]{Woodall1970PLMS}.} such that $L=\{k-1,n-k\}$.
See \cite[Section 3]{LBB2014GC} and \cite[Theorem 8]{Woodall1970PLMS}.
Note that the weights are constant for these three examples.
On the other hand, Bannai, Bannai, and Bannai \cite[Theorem 2.2]{BBB2014DM} showed that there is a tight relative $2$-design in $\mathcal{Q}_n$ with $L=\{2,n/2\}$ for $n\equiv 6 \, (\mathrm{mod} \, 8)$, provided that a  Hadamard matrix of order $n/2+1$ exists.
This construction provides examples in which the weights take two distinct values depending on the shells.
See also \cite{YHG2015DM}.
\end{exam}

\begin{exam}\label{tight relative 4-designs}
Working with the tight $4$-$(23,7,1)$ and $4$-$(23,16,52)$ designs instead of a symmetric $2$-$(n+1,k,\lambda)$ design as in Example \ref{tight relative 2-designs}, we obtain four tight relative $4$-designs in $\mathcal{Q}_{22}$ with constant weight such that
\begin{equation*}
	L\in\bigl\{\{6,7\},\{6,15\},\{7,16\},\{15,16\}\bigr\}.
\end{equation*}
See \cite[Theorem 6.3]{BBZ2017DCC} and \cite[Section 3]{LBB2014GC}.
\end{exam}

Let $(Y,\omega)$ be a tight relative $2e$-design supported on $\bigsqcup_{\ell\in L}X_{\ell}$.
Bannai, Bannai, and Bannai \cite[Theorem 2.1]{BBB2014DM} showed that if the stabilizer of $x$ in the automorphism group of $(X,\mathscr{R})$ acts transitively on each of the shells $X_i$ then $\omega$ is constant on $Y\cap X_{\ell}$ for every $\ell\in L$.
The next theorem generalizes this result by replacing group actions by combinatorial regularity.
Observe that the fibers of the coherent closure of $\bm{T}$ are in general finer than the shells $X_i$.

\begin{thm}\label{constant on fiber}
With reference to Assumption \ref{assume Q-polynomial}, let $(Y,\omega)$ be a tight relative $2e$-design $(e\in\mathbb{N})$ supported on $\bigsqcup_{\ell\in L}X_{\ell}$.
For every $\ell\in L$, the weight $\omega$ is constant on $Y\cap X_{\ell}$ provided that $X_{\ell}$ remains a fiber of the coherent closure of $\bm{T}$.
\end{thm}

\begin{proof}
Let (cf.~\eqref{projection onto primary T-module})
\begin{equation*}
	D=\operatorname{diag}\omega, \qquad \tilde{D}=\operatorname{diag}\omega|_{\bm{T}\hat{x}}=\sum_{\ell\in L}\frac{\langle \omega,\hat{X}_{\ell}\rangle}{|\hat{X}_{\ell}|}E_{\ell}^*.
\end{equation*}
Note that $\tilde{D}\in\bm{T}$.
Let $F$ be the orthogonal projection onto $BV$, where
\begin{equation*}
	B=\sqrt{\!\tilde{D}}\, \sum_{i=0}^eE_i \in\bm{T}.
\end{equation*}
Observe that
\begin{equation*}
	BV=(BB^{\dagger})V,
\end{equation*}
and that $F$ is written as a polynomial in the Hermitian (in fact, real and symmetric) matrix $BB^{\dagger}$.
In particular, $F\in\bm{T}$.

Since $(Y,\omega)$ is tight, we have
\begin{equation*}
	\dim BV=\dim \sqrt{\!\tilde{D}}\!\, \left(\sum_{\ell\in L}E_{\ell}^*\right)\!\!\left(\sum_{i=0}^eE_iV\right)=|Y|.
\end{equation*}
Let $u_1,u_2,\dots,u_{|Y|}$ be an orthonormal basis of $BV$, and let
\begin{equation*}
	G=\left[\begin{array}{c|c|c|c} \!\! u_1\!\! & u_2\!\! & \cdots & \!u_{|Y|}\!\! \end{array}\right].
\end{equation*}
Then we have
\begin{equation}\label{F=GGdagger}
	F=GG^{\dagger}.
\end{equation}

Let
\begin{equation*}
	D'=D|_{Y\times Y}, \quad \tilde{D}'=\tilde{D}|_{Y\times Y}, \quad F'=F|_{Y\times Y}, \quad G'=G|_{Y\times \{1,2,\dots,|Y|\}},
\end{equation*}
where $|_{Y\times Y}$ etc.~mean taking corresponding submatrices.
Note that these are square matrices, and that $D'$ and $\tilde{D}'$ are invertible.
Then it follows that
\begin{equation}\label{get unitary matrix}
	(G')^{\dagger}D'(\tilde{D}')^{-1}G'=I_{|Y|}.
\end{equation}
Indeed, since we may write
\begin{equation*}
	u_i=\sqrt{\!\tilde{D}}\, v_i, \quad \text{where} \ \ v_i\in \sum_{r=0}^eE_rV \qquad (1\leqslant i\leqslant |Y|),
\end{equation*}
it follows from \eqref{Hadamard of flags} (applied to $h=k=e$) and Proposition \ref{akin to Euclidean case} that the $(i,j)$-entry of the LHS in \eqref{get unitary matrix} is equal to
\begin{equation*}
	(v_i)^{\dagger} D v_j = \langle \omega,v_i \circ \overline{v_j} \rangle = \langle \omega|_{\bm{T}\hat{x}},v_i \circ \overline{v_j} \rangle = (v_i)^{\dagger} \tilde{D} v_j =\langle u_j,u_i \rangle =\delta_{i,j},
\end{equation*}
where $\overline{\rule{0ex}{1ex} \ \, }$ means complex conjugate.
By \eqref{F=GGdagger} and \eqref{get unitary matrix}, we have
\begin{equation*}
	I_{|Y|}=D'(\tilde{D}')^{-1}G'(G')^{\dagger}=D'(\tilde{D}')^{-1}F',
\end{equation*}
so that
\begin{equation}\label{compare diagonals}
	(D')^{-1}=(\tilde{D}')^{-1}F'.
\end{equation}
In particular, $F'$ is a diagonal matrix.

Now, let $\ell\in L$ and suppose that $X_{\ell}$ remains a fiber of the coherent closure of $\bm{T}$.
Then the $(y,y)$-entry of $F\in\bm{T}$ is constant for $y\in X_{\ell}$ (cf.~(A\ref{A1}) and (A\ref{A2})), and the same is true for $\tilde{D}$.
Hence it follows from \eqref{compare diagonals} that $\omega(y)=D_{y,y}$ must be constant for $y\in Y\cap X_{\ell}$.
This completes the proof.
\end{proof}

\section{The Terwilliger algebra of \texorpdfstring{$\mathcal{Q}_n$}{Qn}}\label{sec: Terwilliger algebra of hypercube}

For the rest of this paper, we will focus on relative $t$-designs in the $n$-cube $\mathcal{Q}_n$ from Example \ref{Hamming}.
We will need detailed descriptions of the Terwilliger algebra of $\mathcal{Q}_n$ and its irreducible modules, and we collect these in this section.
Thus, we assume that $(X,\mathscr{R})=\mathcal{Q}_n$, where $X=\{0,1\}^n$.
We again fix a base vertex $x\in X$, and write $E_i^*=E_i^*(x)$ $(0\leqslant i\leqslant n)$, $A_1^*=A_1^*(x)$, and $\bm{T}=\bm{T}(x)$.
The $Q$-polynomial ordering we consider is the one given in Example \ref{Hamming}.\footnote{If $n$ is even then $\mathcal{Q}_n$ has another $Q$-polynomial ordering $E_0,E_{n-1},E_2,E_{n-3},\dots$ in terms of the natural ordering; cf.~\cite[p.~305]{BI1984B}.}

\begin{prop}[{cf.~\cite[Section I.C]{Schrijver2005IEEE}}]\label{triple regularity}
We have
\begin{equation}\label{triply regular}
	\bm{T}=\operatorname{span}\{E_i^*A_jE_k^*:0\leqslant i,j,k\leqslant n\}.
\end{equation}
In particular, $\bm{T}$ is a coherent algebra.
\end{prop}

\begin{proof}
The RHS in \eqref{triply regular} is a subspace of $\bm{T}$.
Recall from Example \ref{Hamming} that $\mathcal{Q}_n$ admits the action of $\mathfrak{G}=\mathfrak{S}_2\wr\mathfrak{S}_n$.
The stabilizer $\mathfrak{G}_x$ of $x$ in $\mathfrak{G}$ is isomorphic to $\mathfrak{S}_n$, and it is immediate to see that every orbital of $\mathfrak{G}_x$ is of the form
\begin{equation*}
	\{(y,z)\in X\times X:(x,y)\in R_i,\, (y,z)\in R_j,\, (z,x)\in R_k\}
\end{equation*}
for some $i,j$, and $k$, where the corresponding adjacency matrix is $E_i^*A_jE_k^*$.
Hence the RHS in \eqref{triply regular} agrees with the centralizer algebra for the action of $\mathfrak{G}_x$ on $X$, which is a coherent algebra; cf.~Remark \ref{centralizer algebra}.
Since $\bm{T}$ is generated by the $A_i$ and the $E_i^*$, the result follows.
\end{proof}

\begin{lem}\label{nonzero triple product}
For $0\leqslant i,j,k\leqslant n$, we have $E_i^*A_jE_k^*\ne O$ if and only if
\begin{equation*}
	j\in\bigl\{|i-k|,|i-k|+2,|i-k|+4,\dots,\min\{i+k,2n-i-k\}\bigr\}.
\end{equation*}
\end{lem}

\begin{proof}
Routine.
\end{proof}

Next we recall basic facts about the irreducible $\bm{T}$-modules.
Let $W$ be an irreducible $\bm{T}$-module.
We define the \emph{support} $W_s$, the \emph{endpoint} $r(W)$, and the \emph{diameter} $d(W)$ of $W$ by
\begin{equation*}
	W_s=\{i:E_i^*W\ne 0\}, \qquad r(W)=\min W_s, \qquad d(W)=|W_s|-1,
\end{equation*}
respectively.
We call $W$ \emph{thin} if $\dim E_i^*W\leqslant 1$ $(0\leqslant i\leqslant n)$.

\begin{thm}[{cf.~\cite{Go2002EJC}}]\label{irreducible T-modules}
Let $W$ be an irreducible $\bm{T}$-module and set $r=r(W)$, $r^*=r^*(W)$, $d=d(W)$, and $d^*=d^*(W)$.
Then $W$ is thin, dual thin, and we have
\begin{equation*}
	r=r^*, \qquad d=d^*=n-2r, \qquad W_s=W_s^*=\{r,r+1,\dots,n-r\}.
\end{equation*}
Moreover, the isomorphism class of $W$ is determined by $r$.
\end{thm}

\begin{rem}
Recall that the universal enveloping algebra $U(\mathfrak{sl}_2(\mathbb{C}))$ is defined by the generators $\mathfrak{x},\mathfrak{y},\mathfrak{h}$ and the relations
\begin{equation*}
	\mathfrak{x}\mathfrak{y}-\mathfrak{y}\mathfrak{x}=\mathfrak{h}, \qquad \mathfrak{h}\mathfrak{x}-\mathfrak{x}\mathfrak{h}=2\mathfrak{x}, \qquad \mathfrak{h}\mathfrak{y}-\mathfrak{y}\mathfrak{h}=-2\mathfrak{y}.
\end{equation*}
There is a surjective homomorphism $U(\mathfrak{sl}_2(\mathbb{C}))\rightarrow\bm{T}$ such that (cf.~\cite[Lemma 7.5]{Go2002EJC})
\begin{equation*}
	\mathfrak{x} \mapsto \sum_{i=1}^n E_{i-1}A_1^*E_i, \qquad \mathfrak{y}\mapsto \sum_{i=0}^{n-1} E_{i+1}A_1^*E_i, \qquad \mathfrak{h}\mapsto A_1.
\end{equation*}
Every irreducible $\bm{T}$-module is then irreducible as an $\mathfrak{sl}_2(\mathbb{C})$-module.
We also obtain another surjective homomorphism $U(\mathfrak{sl}_2(\mathbb{C}))\rightarrow\bm{T}$ by interchanging $A_1$ and $A_1^*$ and replacing the $E_i$ by the $E_i^*$ above; cf.~\cite[Lemma 5.3]{Go2002EJC}.
\end{rem}

From now on, we fix an orthogonal irreducible decomposition
\begin{equation}\label{irreducible decomposition}
	V=\bigoplus_{W\in\Lambda}W
\end{equation}
of the standard module $V$.
In view of Theorem \ref{irreducible T-modules}, let
\begin{equation}\label{endpoint r}
	\Lambda_r=\{W\in\Lambda:r(W)=r^*(W)=r\} \qquad (0\leqslant r\leqslant\lfloor n/2 \rfloor),
\end{equation}
and fix a unit vector $v_W\in E_rW$ for each $W\in\Lambda_r$.
Since
\begin{equation}\label{dimension in terms of multiplicities}
	\dim E_iV=\sum_{W\in\Lambda}\dim E_iW=\sum_{r=0}^i |\Lambda_r| \qquad (0\leqslant i\leqslant\lfloor n/2 \rfloor)
\end{equation}
by Theorem \ref{irreducible T-modules}, it follows from \eqref{multiplicities} that
\begin{equation*}
	|\Lambda_r|=\binom{n}{r}-\binom{n}{r-1} \qquad (0\leqslant r\leqslant\lfloor n/2 \rfloor).
\end{equation*}
It is known (cf.~\cite[Theorem 9.2]{Go2002EJC}) that if $W\in\Lambda_r$ then the vectors
\begin{equation}\label{standard basis}
	E_r^*v_W,\, E_{r+1}^*v_W, \dots,\, E_{n-r}^*v_W
\end{equation}
form an orthogonal basis of $W$, called a \emph{standard basis} of $W$.
By \cite[Lemma 6.6]{Go2002EJC}, we also have
\begin{equation}\label{norms}
	\|E_i^*v_W\|^2=\binom{n-2r}{i-r} \|E_r^*v_W\|^2 \qquad (r\leqslant i\leqslant n-r).
\end{equation}
We note that
\begin{equation}\label{norm of initial vector}
	1=\|v_W\|^2=\sum_{i=r}^{n-r}\|E_i^*v_W\|^2=2^{n-2r}\|E_r^*v_W\|^2.
\end{equation}

For $W,W'\in\Lambda_r$, we observe that the linear map $W\rightarrow W'$ defined by
\begin{equation*}
	E_i^*v_W \mapsto E_i^*v_{W'} \qquad (r\leqslant i\leqslant n-r)
\end{equation*}
is an isometric isomorphism of $\bm{T}$-modules.
Let
\begin{equation}\label{definition of matrix units}
	\breve{E}_r^{i,j}=\frac{2^{n-2r}}{\sqrt{\!\binom{n-2r}{i-r}\!\binom{n-2r}{j-r}}}\sum_{W\in\Lambda_r} (E_i^*v_W)(E_j^*v_W)^{\dagger} \qquad (r\leqslant i,j\leqslant n-r).
\end{equation}
Then we have
\begin{equation}\label{hermitian}
	(\breve{E}_r^{i,j})^{\dagger}=\breve{E}_r^{j,i} \qquad (r\leqslant i,j\leqslant n-r),
\end{equation}
and from \eqref{norms} and \eqref{norm of initial vector} it follows that
\begin{equation*}
	\breve{E}_r^{i,j}\breve{E}_{r'}^{i',j'}=\delta_{r,r'}\delta_{j,i'}\breve{E}_r^{i,j'}
\end{equation*}
for $0\leqslant r,r'\leqslant\lfloor n/2 \rfloor$, $r\leqslant i,j\leqslant n-r$, and $r'\leqslant i',j'\leqslant n-r'$.
By Theorem \ref{irreducible T-modules} and Wedderburn's theorem (cf.~\cite[Section 3]{CR1990B}), $\bm{T}$ is isomorphic to the direct sum of full matrix algebras
\begin{equation*}
	\bm{T}\cong \bigoplus_{r=0}^{\lfloor n/2 \rfloor} M_{n-2r+1}(\mathbb{C}),
\end{equation*}
and the $\breve{E}_r^{i,j}$ form an orthogonal basis of $\bm{T}$.
See also \cite[Section 2]{Gijswijt2005D}.
We note that
\begin{equation}\label{second basis for block}
	E_i^*\bm{T}E_j^*=\operatorname{span}\bigl\{\breve{E}_r^{i,j} : 0\leqslant r\leqslant \min\{i,j,n-i,n-j\}\bigr\} \qquad (0\leqslant i, j\leqslant n).
\end{equation}

We now recall the \emph{Hahn polynomials} \cite[Section 1.5]{KS1998R}
\begin{equation}\label{Hahn}
	Q_r(\xi;\alpha,\beta,N)=\hypergeometricseries{3}{2}{-\xi,-r,r+\alpha+\beta+1}{\alpha+1,-N}{1}\in\mathbb{R}[\xi] \quad (0\leqslant r \leqslant N),
\end{equation}
where
\begin{equation*}
	\hypergeometricseries{s}{t}{a_1,\dots,a_s}{b_1,\dots,b_t}{c} = \sum_{i=0}^{\infty}\frac{(a_1)_i\cdots(a_s)_i}{(b_1)_i\cdots(b_t)_i}\frac{c^i}{i!},
\end{equation*}
and
\begin{equation*}
	(a)_i=a(a+1)\cdots(a+i-1).
\end{equation*}
For $\alpha,\beta>-1$, or for $\alpha,\beta<-N$, we have
\begin{align}
	\sum_{\xi=0}^N\binom{\alpha+\xi}{\xi}\!\binom{\beta+N-\xi}{N-\xi} & Q_r(\xi;\alpha,\beta,N) Q_{r'}(\xi;\alpha,\beta,N) \label{orthogonality} \\
	&=\delta_{r,r'} \frac{(-1)^r(r+\alpha+\beta+1)_{N+1}(\beta+1)_rr!}{(2r+\alpha+\beta+1)(\alpha+1)_r(-N)_rN!}. \notag
\end{align}

Our aim is to describe the entries of the $\breve{E}_r^{i,j}$.
In view of \eqref{hermitian}, we will assume for the rest of this section that
\begin{equation*}
	0\leqslant i\leqslant j\leqslant n.
\end{equation*}
By Proposition \ref{triple regularity} and Lemma \ref{nonzero triple product}, we have
\begin{equation*}
	E_i^*\bm{T}E_j^*=\operatorname{span}\bigl\{ E_i^*A_{2\xi+j-i} E_j^* :  0\leqslant \xi\leqslant \min\{i,n-j\}\bigr\}.
\end{equation*}
Moreover, it follows that (cf.~\eqref{second basis for block})
\begin{align}
	E_i^* & A_{2\xi+j-i} E_j^* \label{block-diagonalization} \\
	&=\sum_{r=0}^{\min\{i,n-j\}} \!\!\!\hypergeometricseries{3}{2}{-\xi,-r,r-n-1}{j-n, -i}{1}\!\frac{\binom{j}{i-\xi}\!\binom{n-j}{\xi}\!\binom{j-r}{j-i}\sqrt{\!\binom{n-2r}{j-r}}}{\binom{j}{i}\sqrt{\!\binom{n-2r}{i-r}}} \breve{E}_r^{i,j}. \notag
\end{align}
This formula can be found in \cite[Section 10]{MT2009EJC}.
See also \cite{Schrijver2005IEEE,Vallentin2009LAA} for similar calculations.

If $i\leqslant n-j$ then
\begin{equation*}
	\hypergeometricseries{3}{2}{-\xi,-r,r-n-1}{j-n, -i}{1}=Q_r(\xi;j-n-1,-j-1,i).
\end{equation*}
Since
\begin{equation*}
	\binom{j}{i-\xi}\!\binom{n-j}{\xi}=(-1)^i \binom{j-n-1+\xi}{\xi}\!\binom{-j-1+i-\xi}{i-\xi},
\end{equation*}
it follows from \eqref{orthogonality} (applied to $\alpha=j-n-1$, $\beta=-j-1$, $N=i$) and \eqref{block-diagonalization} that, for $0\leqslant r\leqslant i$,
\begin{align*}
		\sum_{\xi=0}^i &\ \hypergeometricseries{3}{2}{-\xi,-r,r-n-1}{j-n, -i}{1}\! E_i^*A_{2\xi+j-i}E_j^*  \\
		& \qquad = \frac{(-1)^r(r-n-1)_{i+1}(-j)_rr!}{(2r-n-1)(j-n)_r(-i)_ri!} \cdot \frac{(-1)^i\binom{j-r}{j-i}\sqrt{\!\binom{n-2r}{j-r}}}{\binom{j}{i}\sqrt{\!\binom{n-2r}{i-r}}} \breve{E}_r^{i,j} \\
		& \qquad =\frac{\binom{n}{i}\!\binom{n-i}{r}\sqrt{\!\binom{n-2r}{j-r}}}{\left(\!\binom{n}{r}\!-\!\binom{n}{r-1}\!\right)\!\!\binom{n-j}{r}\sqrt{\!\binom{n-2r}{i-r}}} \breve{E}_r^{i,j}.
\end{align*}
Likewise, if $n-j\leqslant i$ then
\begin{equation*}
	\hypergeometricseries{3}{2}{-\xi,-r,r-n-1}{j-n, -i}{1}=Q_r(\xi;-i-1,i-n-1,n-j).
\end{equation*}
In this case, since
\begin{equation*}
	\binom{j}{i-\xi}\!\binom{n-j}{\xi}\!\binom{j}{i}^{\!\!-1}=(-1)^{n-j}\binom{-i-1+\xi}{\xi}\!\binom{i-1-j-\xi}{n-j-\xi}\!\binom{n-i}{n-j}^{\!\!-1},
\end{equation*}
again it follows from \eqref{orthogonality} (applied to $\alpha=-i-1$, $\beta=i-n-1$, $N=n-j$) and \eqref{block-diagonalization} that, for $0\leqslant r\leqslant n-j$,
\begin{align*}
		\sum_{\xi=0}^{n-j}&\ \hypergeometricseries{3}{2}{-\xi,-r,r-n-1}{j-n, -i}{1}\! E_i^*A_{2\xi+j-i}E_j^*  \\
		& \qquad =\frac{(-1)^r(r-n-1)_{n-j+1}(i-n)_rr!}{(2r-n-1)(-i)_r(j-n)_r(n-j)!} \cdot \frac{(-1)^{n-j}\binom{j-r}{j-i}\sqrt{\!\binom{n-2r}{j-r}}}{\binom{n-i}{n-j}\!\sqrt{\!\binom{n-2r}{i-r}}} \breve{E}_r^{i,j} \\
		& \qquad =\frac{\binom{n}{i}\!\binom{n-i}{r}\sqrt{\!\binom{n-2r}{j-r}}}{\left(\!\binom{n}{r}\!-\!\binom{n}{r-1}\!\right)\!\!\binom{n-j}{r}\sqrt{\!\binom{n-2r}{i-r}}} \breve{E}_r^{i,j}.
\end{align*}
In either case, it follows that
\begin{align}
	\breve{E}_r^{i,j}=&\,\frac{\left(\!\binom{n}{r}\!-\!\binom{n}{r-1}\!\right)\!\!\binom{n-j}{r}\sqrt{\!\binom{n-2r}{i-r}}}{\binom{n}{i}\!\binom{n-i}{r}\sqrt{\!\binom{n-2r}{j-r}}} \label{Q} \\
	& \ \ \times \!\! \sum_{\xi=0}^{\min\{i,n-j\}}\!\!\!\!\hypergeometricseries{3}{2}{-\xi,-r,r-n-1}{j-n, -i}{1} \! E_i^*A_{2\xi+j-i}E_j^* \notag
\end{align}
for $0\leqslant i\leqslant j\leqslant n$ and $0\leqslant r\leqslant \min\{i,n-j\}$.

\section{Tight relative \texorpdfstring{$2e$}{2e}-designs on two shells in \texorpdfstring{$\mathcal{Q}_n$}{Qn}}\label{sec: main result}

We retain the notation of the previous sections.
In this section, we discuss tight relative $2e$-designs $(Y,\omega)$ in $\mathcal{Q}_n$ supported on two shells $X_{\ell}\sqcup X_m$, i.e., $L=\{\ell,m\}$ (cf.~\eqref{L}).
Recall from \eqref{explicit lower bound} that we have in this case
\begin{equation*}
	|Y|=\binom{n}{e}+\binom{n}{e-1},
\end{equation*}
but recall also that this is valid under the additional condition that $e\leqslant \ell,m\leqslant n-e$.
However, both $(Y\cap X_{\ell},E_{\ell}^*\omega)$ and $(Y\cap X_m,E_m^*\omega)$ are relative $(2e-1)$-designs by Theorem \ref{Assmus-Mattson}, so that if $\ell<2e$ or $\ell>n-2e$ for example, then $(Y\cap X_{\ell},E_{\ell}^*\omega)$ must be trivial in view of Example \ref{designs in Hamming schemes}, i.e., $X_{\ell}\subset Y$ and $\omega$ is constant on $X_{\ell}$, and hence $(Y\cap X_m,E_m^*\omega)$ is by itself a relative $2e$-design; cf.~Remark \ref{remove trivial part}.
This shows that the above condition is not a restrictive one.
We also note that

\begin{lem}\label{complement}
Let $(Y,\omega)$ be a relative $t$-design in $\mathcal{Q}_n$ supported on $\bigsqcup_{\ell\in L}X_{\ell}$.
Then $(Y',A_n\omega)$ is a relative $t$-design supported on $\bigsqcup_{\ell\in L}X_{n-\ell}$, where $Y'=\{y':y\in Y\}$, and for every $y\in X$, $y'$ denotes the unique vertex such that $(y,y')\in R_n$.
\end{lem}

\begin{proof}
Immediate from $E_iA_n\in\operatorname{span}\{E_i\}$ $(0\leqslant i\leqslant n)$.
\end{proof}

In view of the above comments, we now make the following assumption:

\begin{assump}\label{assumption on Y}
In this section, let $(Y,\omega)$ be a tight relative $2e$-design $(e\in\mathbb{N})$ in $\mathcal{Q}_n$ supported on two shells $X_{\ell}\sqcup X_m$, where
\begin{equation*}
	e\leqslant\ell<m\leqslant n-\ell\,(\leqslant n-e).
\end{equation*}
\end{assump}

Our aim is to show that $Y$ then induces the structure of a coherent configuration with two fibers, and to obtain a necessary condition on the existence of such $(Y,\omega)$ akin to Delsarte's theorem on tight $2e$-designs.
To this end, we first recall the proof of \eqref{explicit lower bound} given in \cite[Theorem 2.7, Example 2.9]{BBST2015EJC} under the above assumption.

For convenience, set
\begin{equation*}
	E_L^*=E_{\ell}^*+E_m^*.
\end{equation*}
By \eqref{irreducible decomposition} and \eqref{endpoint r}, we have
\begin{equation}\label{important subspace}
	E_L^*\!\left(\sum_{i=0}^eE_iV\right)=\sum_{r=0}^e \sum_{W\in\Lambda_r} E_L^*\!\left(\sum_{i=r}^eE_iW\right).
\end{equation}
Let $W\in\Lambda_r$, where $0\leqslant r\leqslant e$.
Recall Theorem \ref{irreducible T-modules} and also the standard basis \eqref{standard basis} of $W$.
If $r=e$ then $E_eW$ is spanned by $v_W$, and hence we have
\begin{equation*}
	E_L^*E_eW=\operatorname{span}\{E_L^*v_W\}.
\end{equation*}
Note that $E_L^*v_W$ is non-zero by Assumption \ref{assumption on Y}, and hence
\begin{equation*}
	\dim E_L^*E_eW=1
\end{equation*}
in this case.
Suppose next that $0\leqslant r<e$.
On the one hand, since
\begin{equation*}
	E_L^*\!\left(\sum_{i=r}^eE_iW\right) \subset E_L^*W=E_{\ell}^*W+E_m^*W,
\end{equation*}
we have
\begin{equation*}
	\dim E_L^*\!\left(\sum_{i=r}^eE_iW\right)\leqslant 2.
\end{equation*}
On the other hand, it follows from \eqref{tridiagonal action} that
\begin{equation}\label{comment 1}
	v_W,A_1^*v_W\in E_rW+E_{r+1}W\subset \sum_{i=r}^eE_iW,
\end{equation}
and hence
\begin{equation*}
	E_L^*v_W,E_L^*A_1^*v_W\in E_L^*\!\left(\sum_{i=r}^eE_iW\right).
\end{equation*}
Moreover, we have (cf.~\eqref{dual adjacency matrix})
\begin{equation*}
	E_L^*v_W=E_{\ell}^*v_W+E_m^*v_W, \qquad E_L^*A_1^*v_W=\theta_{\ell}^*E_{\ell}^*v_W+\theta_m^*E_m^*v_W,
\end{equation*}
so that these two vectors are non-zero and are linearly independent by Assumption \ref{assumption on Y} and since $\theta_{\ell}^*\ne\theta_m^*$.
It follows that
\begin{equation*}
	\dim E_L^*\!\left(\sum_{i=r}^eE_iW\right)=2.
\end{equation*}
Note that in this case we in fact have
\begin{equation*}
	E_L^*\!\left(\sum_{i=r}^eE_iW\right)=\operatorname{span}\{E_{\ell}^*v_W,E_m^*v_W\},
\end{equation*}
as
\begin{equation}\label{comment 2}
	E_{\ell}^*v_W=E_L^*\frac{\theta_m^*I-A_1^*}{\theta_m^*-\theta_{\ell}^*}v_W, \qquad E_m^*v_W=E_L^*\frac{\theta_{\ell}^*I-A_1^*}{\theta_{\ell}^*-\theta_m^*}v_W.
\end{equation}
Combining these comments, we now obtain \eqref{explicit lower bound} as follows:
\begin{align*}
	\dim E_L^*\!\left(\sum_{i=0}^eE_iV\right) &= \sum_{r=0}^e \sum_{W\in\Lambda_r} \dim E_L^*\!\left(\sum_{i=r}^eE_iW\right) \\
	&= |\Lambda_e|+\sum_{r=0}^{e-1} 2|\Lambda_r| \\
	&= \dim E_eV+\dim E_{e-1}V \\
	&= \binom{n}{e}+\binom{n}{e-1},
\end{align*}
where we have used \eqref{multiplicities} and \eqref{dimension in terms of multiplicities}.

By the above discussions, the set of vectors below forms an orthogonal basis of the subspace \eqref{important subspace}:
\begin{equation*}
	\left(\bigsqcup_{r=0}^{e-1}\bigsqcup_{W\in\Lambda_r}\!\{E_{\ell}^*v_W,E_m^*v_W\}\right)\bigsqcup\left(\bigsqcup_{W\in\Lambda_e}\!\{E_L^*v_W\}\right).
\end{equation*}
As in the proof of Theorem \ref{constant on fiber}, let
\begin{equation*}
	D=\operatorname{diag}\omega.
\end{equation*}
We next apply $\sqrt{\!D}$ to the above basis vectors and compute their inner products.
First, let $W,W'\in\bigsqcup_{r=0}^{e-1}\Lambda_r$.
It is clear that
\begin{equation}\label{orthogonal 1}
	\left\langle \sqrt{\!D}E_{\ell}^*v_W,\sqrt{\!D}E_m^*v_{W'}\right\rangle=\left\langle \sqrt{\!D}E_m^*v_W,\sqrt{\!D}E_{\ell}^*v_{W'}\right\rangle=0.
\end{equation}
By \eqref{comment 2}, we have
\begin{equation*}
	(E_{\ell}^*\overline{v_W})\circ(E_{\ell}^*v_{W'})=E_L^*u,
\end{equation*}
where $\overline{\rule{0ex}{1ex} \ \, }$ means complex conjugate, and
\begin{equation*}
	u=\left(\frac{\theta_m^*I-A_1^*}{\theta_m^*-\theta_{\ell}^*}\overline{v_W}\right)\!\circ\!\left(\frac{\theta_m^*I-A_1^*}{\theta_m^*-\theta_{\ell}^*}v_{W'}\right).
\end{equation*}
Observe that $u$ belongs to $\sum_{i=0}^{2e}E_iV$ by \eqref{Hadamard of flags} (applied to $h=k=e$) and \eqref{comment 1}.
Hence, by Proposition \ref{akin to Euclidean case} we have
\begin{align}
	\left\langle \sqrt{\!D}E_{\ell}^*v_W,\sqrt{\!D}E_{\ell}^*v_{W'}\right\rangle&=\langle\omega,E_L^*u\rangle \label{orthogonal 2} \\
	&=\langle\omega,u\rangle \notag \\
	&=\frac{\langle \omega,\hat{X}_{\ell}\rangle}{|X_{\ell}|}\langle \hat{X}_{\ell},u\rangle+\frac{\langle \omega,\hat{X}_m\rangle}{|X_m|}\langle \hat{X}_m,u\rangle \notag \\
	&=\frac{\langle \omega,\hat{X}_{\ell}\rangle}{|X_{\ell}|}\langle \hat{X}_{\ell},E_L^*u\rangle \notag \\
	&=\frac{\langle \omega,\hat{X}_{\ell}\rangle}{|X_{\ell}|}\langle E_{\ell}^*v_W,E_{\ell}^*v_{W'}\rangle \notag \\
	&=\delta_{W,W'}\,\frac{\langle \omega,\hat{X}_{\ell}\rangle}{|X_{\ell}|}\|E_{\ell}^*v_W\|^2. \notag
\end{align}
Likewise, we have
\begin{equation}\label{orthogonal 3}
	\left\langle \sqrt{\!D}E_m^*v_W,\sqrt{\!D}E_m^*v_{W'}\right\rangle=\delta_{W,W'}\,\frac{\langle \omega,\hat{X}_m\rangle}{|X_m|}\|E_m^*v_W\|^2.
\end{equation}
Next, let $W\in\bigsqcup_{r=0}^{e-1}\Lambda_r$ and $W'\in\Lambda_e$.
Then, by the same argument we have
\begin{equation}\label{orthogonal 4}
	\left\langle \sqrt{\!D}E_{\ell}^*v_W,\sqrt{\!D}E_L^*v_{W'}\right\rangle=\left\langle \sqrt{\!D}E_m^*v_W,\sqrt{\!D}E_L^*v_{W'}\right\rangle=0.
\end{equation}
Finally, let $W,W'\in\Lambda_e$.
In this case, we have
\begin{align}
	\left\langle \sqrt{\!D}E_L^*v_W,\right.\!\! & \left.\sqrt{\!D}E_L^*v_{W'}\right\rangle \label{orthogonal 5} \\
	&=\delta_{W,W'}\!\left(\frac{\langle \omega,\hat{X}_{\ell}\rangle}{|\hat{X}_{\ell}|}\|E_{\ell}^*v_W\|^2+\frac{\langle \omega,\hat{X}_m\rangle}{|\hat{X}_m|}\|E_m^*v_W\|^2\right). \notag
\end{align}
Since $(Y,\omega)$ is a tight relative $2e$-design, it follows from \eqref{orthogonal 1}--\eqref{orthogonal 5} that the set of vectors below is an orthogonal basis of the subspace $\sqrt{\!D}V=\operatorname{span}\{\hat{y}:y\in Y\}$ of dimension $|Y|=\binom{n}{e}+\binom{n}{e-1}$:
\begin{equation*}
	\left(\bigsqcup_{r=0}^{e-1}\bigsqcup_{W\in\Lambda_r}\!\!\left\{\sqrt{\!D}E_{\ell}^*v_W,\sqrt{\!D}E_m^*v_W\right\}\right)\bigsqcup\left(\bigsqcup_{W\in\Lambda_e}\!\!\left\{\sqrt{\!D}E_L^*v_W\right\}\right).
\end{equation*}

For convenience, set
\begin{equation*}
	Y_{\ell}=Y\cap X_{\ell}, \qquad Y_m=Y\cap X_m.
\end{equation*}
We will naturally make the following identification by discarding irrelevant entries:
\begin{gather*}
	\sqrt{\!D}E_{\ell}^*V=\operatorname{span}\{\hat{y}:y\in Y_{\ell}\} \ \longleftrightarrow \ \mathbb{C}^{Y_{\ell}}, \\
	\sqrt{\!D}E_m^*V=\operatorname{span}\{\hat{y}:y\in Y_m\} \ \longleftrightarrow \ \mathbb{C}^{Y_m}.
\end{gather*}
We write
\begin{equation*}
	\Lambda_r=\bigl\{W_r^1,W_r^2,\dots,W_r^{|\Lambda_r|}\bigr\} \qquad (0\leqslant r\leqslant e).
\end{equation*}
For $0\leqslant r\leqslant e$, define a $|Y_{\ell}|\times |\Lambda_r|$ matrix $H_r^{\ell}$ and a $|Y_m|\times |\Lambda_r|$ matrix $H_r^m$ by
\begin{align*}
	H_r^{\ell}&=\left[\begin{array}{c|c|c} \!\!\sqrt{\!D}E_{\ell}^*v_{W_r^1}\! & \cdots & \!\sqrt{\!D}E_{\ell}^*v_{W_r^{|\Lambda_r|}}\!\! \end{array}\right], \\
	H_r^m&=\left[\begin{array}{c|c|c} \!\!\sqrt{\!D}E_m^*v_{W_r^1}\! & \cdots & \!\sqrt{\!D}E_m^*v_{W_r^{|\Lambda_r|}}\!\! \end{array}\right].
\end{align*}
We then define a \emph{characteristic matrix} $H$ of $(Y,\omega)$ by
\begin{equation*}
	H=\left[\begin{array}{c|c|c|c|c|c|c} \!\!H_0^{\ell}\! & \cdots & \!H_{e-1}^{\ell}\! & \!O\! & \cdots & \!O\! & \!H_e^{\ell}\!\! \\[1.2pt] \hline \rule{0pt}{11pt} \!\!O\! & \cdots & \!O\! & \!H_0^m\! & \cdots & \!H_{e-1}^m\! & \!H_e^m\!\! \end{array}\right].
\end{equation*}
We note that $H$ is a square matrix of size $|Y|=\binom{n}{e}+\binom{n}{e-1}$.
By \eqref{norms}, \eqref{norm of initial vector}, and \eqref{orthogonal 1}--\eqref{orthogonal 5}, and since
\begin{equation*}
	|X_i|=\binom{n}{i} \qquad (0\leqslant i\leqslant n),
\end{equation*}
we have
\begin{equation}\label{unitary}
	H^{\dagger}H=\left(\mathop{\oplus}_{r=0}^{e-1}\kappa_r^{\ell} I_{|\Lambda_r|}\right)\!\oplus\!\left(\mathop{\oplus}_{r=0}^{e-1}\kappa_r^m I_{|\Lambda_r|}\right)\!\oplus\kappa_e I_{|\Lambda_e|},
\end{equation}
where
\begin{gather*}
	\kappa_r^{\ell}=\frac{\omega_{\ell}\binom{n-2r}{\ell-r}}{2^{n-2r}\binom{n}{\ell}}	, \ \ \kappa_r^m=\frac{\omega_m\binom{n-2r}{m-r}}{2^{n-2r}\binom{n}{m}} \qquad (0\leqslant r<e), \\
	\kappa_e=\frac{\omega_{\ell}\binom{n-2e}{\ell-e}}{2^{n-2e}\binom{n}{\ell}}+\frac{\omega_m\binom{n-2e}{m-e}}{2^{n-2e}\binom{n}{m}},
\end{gather*}
and we abbreviate
\begin{equation*}
	\omega_{\ell}=\langle\omega,\hat{X}_{\ell}\rangle, \qquad \omega_m=\langle\omega,\hat{X}_m\rangle.
\end{equation*}
Let $K$ denote the diagonal matrix on the RHS in \eqref{unitary}.
Then it follows that
\begin{align}
	I_{|Y|}&=HK^{-1}H^{\dagger} \label{three equations} \\
	&= \left[\begin{array}{c|c} \sum_{r=0}^e \frac{1}{\kappa_r^{\ell}} H_r^{\ell} (H_r^{\ell})^{\dagger} & \frac{1}{\kappa_e} H_e^{\ell} (H_e^m)^{\dagger} \\[1.2pt] \hline \rule{0pt}{11pt}\frac{1}{\kappa_e} H_e^m (H_e^{\ell})^{\dagger} & \sum_{r=0}^e \frac{1}{\kappa_r^m} H_r^m (H_r^m)^{\dagger} \end{array} \right], \notag
\end{align}
where we write
\begin{equation}\label{temporary notation}
	\kappa_e^{\ell}=\kappa_e^m:=\kappa_e
\end{equation}
for brevity.
In particular, we have
\begin{equation}\label{idempotents lm1}
	\frac{1}{\kappa_e} H_e^{\ell} (H_e^m)^{\dagger}=O.
\end{equation}
Moreover, from \eqref{unitary} and \eqref{three equations} it follows that
\begin{align}
	\left(\frac{1}{\kappa_r^{\ell}} H_r^{\ell}(H_r^{\ell})^{\dagger}\right)\!\!\left(\frac{1}{\kappa_{r'}^{\ell}} H_{r'}^{\ell}(H_{r'}^{\ell})^{\dagger}\right) \! &=\delta_{r,r'}\frac{1}{\kappa_r^{\ell}} H_r^{\ell}(H_r^{\ell})^{\dagger} \qquad (0\leqslant r,r'<e), \label{idempotents l1} \\
	\frac{1}{\kappa_e} H_e^{\ell}(H_e^{\ell})^{\dagger}&=I_{|Y_{\ell}|}-\sum_{r=0}^{e-1} \frac{1}{\kappa_r^{\ell}} H_r^{\ell}(H_r^{\ell})^{\dagger}, \label{idempotents l2} \\
	\left(\frac{1}{\kappa_r^m} H_r^m(H_r^m)^{\dagger}\right)\!\!\left(\frac{1}{\kappa_{r'}^m} H_{r'}^m(H_{r'}^m)^{\dagger}\right) \! &=\delta_{r,r'}\frac{1}{\kappa_r^m} H_r^m(H_r^m)^{\dagger} \qquad (0\leqslant r,r'<e), \label{idempotents m1} \\
	\frac{1}{\kappa_e} H_e^m(H_e^m)^{\dagger}&=I_{|Y_m|}-\sum_{r=0}^{e-1} \frac{1}{\kappa_r^m} H_r^m(H_r^m)^{\dagger}. \label{idempotents m2}
\end{align}
Note that the matrices $(\kappa_r^{\ell})^{-1}H_r^{\ell}(H_r^{\ell})^{\dagger},(\kappa_r^m)^{-1} H_r^m(H_r^m)^{\dagger}$ ($0\leqslant r<e$) are non-zero since $H_r^{\ell},H_r^m$ are non-zero.
Likewise, by setting
\begin{equation*}
	\kappa_r=\sqrt{\kappa_r^{\ell}\kappa_r^m} \qquad (0\leqslant r<e)
\end{equation*}
for brevity, we have
\begin{align}
	\left(\frac{1}{\kappa_r^{\ell}} H_r^{\ell}(H_r^{\ell})^{\dagger}\right)\!\!\left(\frac{1}{\kappa_{r'}} H_{r'}^{\ell}(H_{r'}^m)^{\dagger}\right) \! &=\delta_{r,r'}\frac{1}{\kappa_r} H_r^{\ell}(H_r^m)^{\dagger} \qquad (0\leqslant r,r'<e), \label{idempotents lm2} \\
	\left(\frac{1}{\kappa_r} H_r^{\ell}(H_r^m)^{\dagger}\right)\!\!\left(\frac{1}{\kappa_{r'}^m} H_{r'}^m(H_{r'}^m)^{\dagger}\right) \! &=\delta_{r,r'}\frac{1}{\kappa_r} H_r^{\ell}(H_r^m)^{\dagger} \qquad (0\leqslant r,r'<e), \label{idempotents lm3} \\
	\left(\frac{1}{\kappa_r} H_r^{\ell}(H_r^m)^{\dagger}\right)\!\!\left(\frac{1}{\kappa_{r'}} H_{r'}^m(H_{r'}^{\ell})^{\dagger}\right) \! &=\delta_{r,r'}\frac{1}{\kappa_r^{\ell}} H_r^{\ell}(H_r^{\ell})^{\dagger} \qquad (0\leqslant r,r'<e), \label{idempotents lm4} \\
	\left(\frac{1}{\kappa_r} H_r^m(H_r^{\ell})^{\dagger}\right)\!\!\left(\frac{1}{\kappa_{r'}} H_{r'}^{\ell}(H_{r'}^m)^{\dagger}\right) \! &=\delta_{r,r'}\frac{1}{\kappa_r^m} H_r^m(H_r^m)^{\dagger} \qquad (0\leqslant r,r'<e). \label{idempotents lm5}
\end{align}
Since the matrices $(\kappa_r^{\ell})^{-1}H_r^{\ell}(H_r^{\ell})^{\dagger}, (\kappa_r^m)^{-1} H_r^m(H_r^m)^{\dagger}$ ($0\leqslant r<e$) are non-zero, it follows from \eqref{idempotents lm2}--\eqref{idempotents lm5} that the matrices $(\kappa_r)^{-1} H_r^{\ell}(H_r^m)^{\dagger}$ $(0\leqslant r<e)$ are non-zero and are linearly independent.

It follows from Theorem \ref{constant on fiber} and Proposition \ref{triple regularity} that $\omega$ is constant on each of $Y_{\ell}$ and $Y_m$, from which it follows that
\begin{equation}\label{values of omega}
	D_{y,y}=\omega(y)=\begin{cases} \dfrac{\omega_{\ell}}{|Y_{\ell}|} & \text{if} \ y\in Y_{\ell}, \\[3mm] \dfrac{\omega_m}{|Y_m|} & \text{if} \ y\in Y_m. \end{cases}
\end{equation}
Hence, by comparing with the formula \eqref{definition of matrix units} for the matrices $\breve{E}_r^{i,j}$, we have
\begin{align}
	\frac{1}{\kappa_r^{\ell}} H_r^{\ell}(H_r^{\ell})^{\dagger}&=\frac{\binom{n}{\ell}}{|Y_{\ell}|}\breve{E}_r^{\ell,\ell}|_{Y_{\ell}\times Y_{\ell}} \qquad (0\leqslant r<e), \label{restriction l1} \\
	\frac{1}{\kappa_e} H_e^{\ell}(H_e^{\ell})^{\dagger}&=\frac{\omega_{\ell}\binom{n-2e}{\ell-e}}{2^{n-2e}\kappa_e|Y_{\ell}|}\breve{E}_e^{\ell,\ell}|_{Y_{\ell}\times Y_{\ell}}, \label{restriction l2} \\
	\frac{1}{\kappa_r^m} H_r^m(H_r^m)^{\dagger}&=\frac{\binom{n}{m}}{|Y_m|}\breve{E}_r^{m,m}|_{Y_m\times Y_m} \qquad (0\leqslant r<e), \label{restriction m1} \\
	\frac{1}{\kappa_e} H_e^m(H_e^m)^{\dagger}&=\frac{\omega_m\binom{n-2e}{m-e}}{2^{n-2e}\kappa_e|Y_m|}\breve{E}_e^{m,m}|_{Y_m\times Y_m}, \label{restriction m2} \\
	\frac{1}{\kappa_r} H_r^{\ell}(H_r^m)^{\dagger}&=\frac{\sqrt{\!\binom{n}{\ell}\!\binom{n}{m}}}{\sqrt{|Y_{\ell}||Y_m|}}\breve{E}_r^{\ell,m}|_{Y_{\ell}\times Y_m} \qquad (0\leqslant r<e), \label{restriction lm1} \\
	\frac{1}{\kappa_e} H_e^{\ell}(H_e^m)^{\dagger}&=\frac{\sqrt{\omega_{\ell}\omega_m\binom{n-2e}{\ell-e}\!\binom{n-2e}{m-e}}}{2^{n-2e}\kappa_e\sqrt{|Y_{\ell}||Y_m|}}\breve{E}_e^{\ell,m}|_{Y_{\ell}\times Y_m}, \label{restriction lm2}
\end{align}
where $|_{Y_{\ell}\times Y_{\ell}}$ etc. mean taking corresponding submatrices.
From \eqref{restriction l2} and \eqref{restriction m2} it follows that the matrices $(\kappa_e)^{-1} H_e^{\ell}(H_e^{\ell})^{\dagger}, (\kappa_e)^{-1} H_e^m(H_e^m)^{\dagger}$ are also non-zero, since each of $\breve{E}_e^{\ell,\ell}|_{Y_{\ell}\times Y_{\ell}},\breve{E}_e^{m,m}|_{Y_m\times Y_m}$ has non-zero constant diagonal entries by \eqref{Q}.

Let $\bm{H}'$ be the set consisting of the $|Y|\times|Y|$ matrices of the form
\begin{equation*}
	\left[\begin{array}{c|c} \sum_{r=0}^e a_r^{\ell,\ell}\frac{1}{\kappa_r^{\ell}} H_r^{\ell}(H_r^{\ell})^{\dagger} & \sum_{r=0}^{e-1}a_r^{\ell,m}\frac{1}{\kappa_r} H_r^{\ell}(H_r^m)^{\dagger} \\[3pt] \hline \rule{0pt}{11pt} \sum_{r=0}^{e-1}a_r^{m,\ell}\frac{1}{\kappa_r} H_r^m(H_r^{\ell})^{\dagger} & \sum_{r=0}^e a_r^{m,m}\frac{1}{\kappa_r^m} H_r^m(H_r^m)^{\dagger} \end{array}\right],
\end{equation*}
where $a_r^{\ell,\ell}$ etc.~are in $\mathbb{C}$, and we are again using the notation \eqref{temporary notation}.
By \eqref{idempotents l1}--\eqref{idempotents lm5} and the above comments, $\bm{H}'$ is a $\mathbb{C}$-algebra with
\begin{equation}\label{dim H'}
	\dim\bm{H}'=4e+2.
\end{equation}

Define
\begin{equation*}
	S_{\ell,\ell}(Y)=\bigl\{j:R_j\cap(Y_{\ell}\times Y_{\ell})\ne\emptyset\bigr\},
\end{equation*}
and define $S_{\ell,m}(Y)(=S_{m,\ell}(Y))$ and $S_{m,m}(Y)$ in the same manner.
Let $\bm{H}$ be the set consisting of the $|Y|\times|Y|$ matrices of the form
\begin{equation}\label{H}
	\left[\begin{array}{c|c} \sum_{j\in S_{\ell,\ell}(Y)} b_j^{\ell,\ell}A_j|_{Y_{\ell}\times Y_{\ell}} & \sum_{j\in S_{\ell,m}(Y)} b_j^{\ell,m}A_j|_{Y_{\ell}\times Y_m} \\[3pt] \hline \rule{0pt}{11pt} \sum_{j\in S_{m,\ell}(Y)} b_j^{m,\ell}A_j|_{Y_m\times Y_{\ell}} & \sum_{j\in S_{m,m}(Y)} b_j^{m,m}A_j|_{Y_m\times Y_m} \end{array}\right],
\end{equation}
where $b_j^{\ell,\ell}$ etc.~are in $\mathbb{C}$.
Then $\bm{H}$ is a $\mathbb{C}$-vector space with
\begin{equation}\label{dim H}
	\dim\bm{H}=|S_{\ell,\ell}(Y)|+|S_{\ell,m}(Y)|+|S_{m,\ell}(Y)|+|S_{m,m}(Y)|.
\end{equation}
Note that $\bm{H}$ is closed under $\circ$.
By \eqref{restriction l1}--\eqref{restriction lm1} and Proposition \ref{triple regularity} (or \eqref{Q}), $\bm{H}'$ is a subspace of $\bm{H}$.

By \eqref{Q}, \eqref{idempotents l2}, \eqref{restriction l1}, and \eqref{restriction l2}, we have
\begin{align}
	I_{|Y_{\ell}|}&=\sum_{r=0}^{e-1} \frac{\binom{n}{\ell}}{|Y_{\ell}|}\breve{E}_r^{\ell,\ell}|_{Y_{\ell}\times Y_{\ell}}+\frac{\omega_{\ell}\binom{n-2e}{\ell-e}}{2^{n-2e}\kappa_e|Y_{\ell}|}\breve{E}_e^{\ell,\ell}|_{Y_{\ell}\times Y_{\ell}} \label{condition on diagonal block} \\
	&=\frac{1}{|Y_{\ell}|}\!\sum_{\xi=0}^{\min\{\ell,n-\ell\}}\!\!\left( \rule{0pt}{22pt} \sum_{r=0}^{e-1}\left(\!\binom{n}{r}\!-\!\binom{n}{r-1}\!\right)\hypergeometricseries{3}{2}{-\xi,-r,r-n-1}{\ell-n, -\ell}{1} \right. \notag \\
	& \qquad \qquad \left. +\frac{\omega_{\ell}\binom{n-2e}{\ell-e}\!\!\left(\!\binom{n}{e}\!-\!\binom{n}{e-1}\!\right)}{2^{n-2e}\kappa_e\binom{n}{\ell}} \hypergeometricseries{3}{2}{-\xi,-e,e-n-1}{\ell-n, -\ell}{1} \!\!\right)\! A_{2\xi}|_{Y_{\ell}\times Y_{\ell}}. \notag
\end{align}
Hence it follows that $\{\xi\ne 0:2\xi\in S_{\ell,\ell}(Y)\}$ is a set of zeros of the polynomial
\begin{align}
	\psi_e^{\ell,\ell}(\xi)=&\,\sum_{r=0}^{e-1}\left(\!\binom{n}{r}\!-\!\binom{n}{r-1}\!\right)\hypergeometricseries{3}{2}{-\xi,-r,r-n-1}{\ell-n,-\ell}{1} \label{Wilson l} \\
	& +\frac{\omega_{\ell}\binom{n-2e}{\ell-e}\!\!\left(\!\binom{n}{e}\!-\!\binom{n}{e-1}\!\right)}{2^{n-2e}\kappa_e\binom{n}{\ell}} \hypergeometricseries{3}{2}{-\xi,-e,e-n-1}{\ell-n, -\ell}{1} \in\mathbb{R}[\xi]. \notag
\end{align}
Note that $\psi_e^{\ell,\ell}(\xi)$ has degree exactly $e$, from which it follows that
\begin{equation}\label{Sll}
	|S_{\ell,\ell}(Y)|\leqslant e+1.
\end{equation}
Likewise, we find that $\{\xi\ne 0:2\xi\in S_{m,m}(Y)\}$ is a set of zeros of the polynomial
\begin{align}
	\psi_e^{m,m}(\xi)=&\,\sum_{r=0}^{e-1}\left(\!\binom{n}{r}\!-\!\binom{n}{r-1}\!\right)\hypergeometricseries{3}{2}{-\xi,-r,r-n-1}{m-n,-m}{1} \label{Wilson m} \\
	& +\frac{\omega_m\binom{n-2e}{m-e}\!\!\left(\!\binom{n}{e}\!-\!\binom{n}{e-1}\!\right)}{2^{n-2e}\kappa_e\binom{n}{m}} \hypergeometricseries{3}{2}{-\xi,-e,e-n-1}{m-n, -m}{1} \in\mathbb{R}[\xi], \notag
\end{align}
and hence that
\begin{equation}\label{Smm}
	|S_{m,m}(Y)|\leqslant e+1.
\end{equation}
Finally, by \eqref{Q}, \eqref{idempotents lm1}, and \eqref{restriction lm2}, we have
\begin{align*}
	O&=\frac{\sqrt{\omega_{\ell}\omega_m\binom{n-2e}{\ell-e}\!\binom{n-2e}{m-e}}}{2^{n-2e}\kappa_e\sqrt{|Y_{\ell}||Y_m|}}\breve{E}_e^{\ell,m}|_{Y_{\ell}\times Y_m} \\
	&=\frac{\sqrt{\omega_{\ell}\omega_m}\binom{n-m}{e}\!\binom{n-2e}{\ell-e}\!\!\left(\!\binom{n}{e}\!-\!\binom{n}{e-1}\!\right)}{2^{n-2e}\kappa_e\sqrt{|Y_{\ell}||Y_m|}\,\binom{n}{\ell}\!\binom{n-\ell}{e}} \\
	& \qquad \times \sum_{\xi=0}^{\min\{\ell,n-m\}}\!\!\!\!\hypergeometricseries{3}{2}{-\xi,-e,e-n-1}{m-n, -\ell}{1} \! A_{2\xi+m-\ell}|_{Y_{\ell}\times Y_m}.
\end{align*}
Hence it follows that $\{\xi:2\xi+m-\ell\in S_{\ell,m}(Y)\}$ is a set of zeros of the polynomial
\begin{equation}\label{Wilson lm}
	\psi_e^{\ell,m}(\xi)=\hypergeometricseries{3}{2}{-\xi,-e,e-n-1}{m-n, -\ell}{1}\in\mathbb{R}[\xi],
\end{equation}
and that
\begin{equation}\label{Slm}
	|S_{\ell,m}(Y)|=|S_{m,\ell}(Y)|\leqslant e.
\end{equation}
By \eqref{dim H}, \eqref{Sll}, \eqref{Smm}, and \eqref{Slm}, we have
\begin{equation*}
	\dim\bm{H}\leqslant 4e+2.
\end{equation*}
Since $\bm{H}'$ is a subspace of $\bm{H}$, it follows from \eqref{dim H'} that $\bm{H}=\bm{H}'$.
In particular, $\bm{H}$ is a $\mathbb{C}$-algebra.
It is also clear that $\bm{H}$ is closed under $\dagger$ and contains $J_{|Y|}$.
We now conclude that $\bm{H}$ is a coherent algebra.
Note also that equality holds in each of \eqref{Sll}, \eqref{Smm}, and \eqref{Slm}.

To summarize:

\begin{thm}\label{counterpart of Wilson's theorem}
Recall Assumption \ref{assumption on Y}.
With the above notation, the following hold:
\begin{enumerate}
\item The set $\bm{H}$ from \eqref{H} is a coherent algebra of type $\left[\begin{smallmatrix} e+1 & e \\ e & e+1 \end{smallmatrix}\right]$.
\item The sets of zeros of the polynomials $\psi_e^{\ell,\ell}(\xi),\psi_e^{m,m}(\xi)$, and $\psi_e^{\ell,m}(\xi)$ from \eqref{Wilson l}, \eqref{Wilson m}, and \eqref{Wilson lm} are given respectively by
\begin{gather*}
	\{\xi\ne 0:2\xi\in S_{\ell,\ell}(Y)\}, \ \{\xi\ne 0:2\xi\in S_{m,m}(Y)\}, \ \text{and} \\ \{\xi:2\xi+m-\ell\in S_{\ell,m}(Y)\}.
\end{gather*}
In particular, the zeros of these polynomials are integral.
\end{enumerate}
\end{thm}

Concerning the scalars $\omega_{\ell}$ and $\omega_m$ appearing in the polynomials $\psi_e^{\ell,\ell}(\xi)$ and $\psi_e^{m,m}(\xi)$, it follows that

\begin{prop}\label{ratio of weights}
Recall Assumption \ref{assumption on Y}.
The scalars $\omega_{\ell}$ and $\omega_m$ satisfies
\begin{equation*}
	\frac{\omega_m}{\omega_{\ell}}=\frac{\binom{n}{m}\binom{n-2e}{\ell-e}}{\binom{n}{\ell}\binom{n-2e}{m-e}}\cdot\frac{|Y_m|-\binom{n}{e-1}}{|Y_{\ell}|-\binom{n}{e-1}}.
\end{equation*}
In particular, the weight function $\omega$ is unique up to a scalar multiple.
\end{prop}

\begin{proof}
By comparing the diagonal entries of both sides in \eqref{condition on diagonal block}, we have
\begin{equation*}
	1=\frac{\psi_e^{\ell,\ell}(0)}{|Y_{\ell}|}=\frac{1}{|Y_{\ell}|}\left(\! \binom{n}{e-1}+\frac{\omega_{\ell}\binom{n-2e}{\ell-e}\!\!\left(\!\binom{n}{e}\!-\!\binom{n}{e-1}\!\right)}{2^{n-2e}\kappa_e\binom{n}{\ell}}\right).
\end{equation*}
Likewise,
\begin{equation*}
	1=\frac{\psi_e^{m,m}(0)}{|Y_m|}=\frac{1}{|Y_m|}\left(\! \binom{n}{e-1}+\frac{\omega_m\binom{n-2e}{m-e}\!\!\left(\!\binom{n}{e}\!-\!\binom{n}{e-1}\!\right)}{2^{n-2e}\kappa_e\binom{n}{m}}\right).
\end{equation*}
By eliminating $\kappa_e$, we obtain the formula for $\omega_m(\omega_{\ell})^{-1}$.
The uniqueness of $\omega$ follows from this and \eqref{values of omega}.
\end{proof}

\begin{exam}
Suppose that $e=1$.
In this case, Theorem \ref{counterpart of Wilson's theorem}\,(i) was previously obtained by Bannai, Bannai, and Bannai \cite[Theorem 2.2\,(i)]{BBB2014DM}.
Moreover, Theorem \ref{counterpart of Wilson's theorem}\,(ii) and Proposition \ref{ratio of weights} are together equivalent to \cite[Proposition 4.3]{BBB2014DM}.
\end{exam}

\begin{exam}
Suppose that $e=2$.
Then we have
\begin{align*}
	\psi_e^{\ell,m}(\xi)&=1+\frac{(-\xi)(-2)(1-n)}{(m-n)(-\ell)}+\frac{(-\xi)(1-\xi)(-2)(-1)(1-n)(2-n)}{(m-n)(m-n+1)(-\ell)(1-\ell)2} \\
	&=1-\frac{2(n-1)\xi}{(n-m)\ell}+\frac{(n-1)(n-2)\xi(\xi-1)}{(n-m)(n-m-1)\ell(\ell-1)}.
\end{align*}
From Example \ref{tight relative 4-designs} we find two parameter sets satisfying Assumption \ref{assumption on Y}:
\begin{center}
\begin{tabular}{>{$}c<{$}>{$}c<{$}>{$}c<{$}>{$}c<{$}}
	\hline\hline
	n & \ell & m & \xi \\
	\hline
	22 & 6 & 7 & 3,5 \\
	\hline
	22 & 6 & 15 & 1,3 \\
	\hline\hline
\end{tabular}
\end{center}
The zeros $\xi$ given in the last column are indeed integers.
Note that the other two parameter sets in Example \ref{tight relative 4-designs} correspond to the complements of these two; cf.~Lemma \ref{complement}.
On the other hand, the existence of tight relative $4$-designs with the following feasible parameter sets was left open in \cite[Section 6]{BBZ2017DCC}:
\begin{center}
\begin{tabular}{>{$}c<{$}>{$}c<{$}>{$}c<{$}>{$}c<{$}}
	\hline\hline
	n & \ell & m & \xi \\
	\hline
	37 & 9 & 16 & \rule{0pt}{10pt} \frac{1}{14} (71\pm\sqrt{337}) \\[1pt]
	\hline
	37 & 9 & 21 & \rule{0pt}{10pt} \frac{1}{14} (55\pm\sqrt{337}) \\[1pt]
	\hline
	41 & 15 & 16 & \rule{0pt}{10pt} \frac{1}{26} (237\pm\sqrt{1569}) \\[1pt]
	\hline
	41 & 15 & 25 & \rule{0pt}{10pt} \frac{1}{26} (153\pm\sqrt{1569}) \\[1pt]
	\hline\hline
\end{tabular}
\end{center}
Here, we are again taking Lemma \ref{complement} into account.
Observe that the zeros $\xi$ are irrational, thus proving the non-existence.
\end{exam}

We end this section with a comment on the expressions of the polynomials $\psi_e^{\ell,\ell}(\xi)$ and $\psi_e^{m,m}(\xi)$.
We first invoke the following identity which agrees with the formula of the backward shift operator on the dual Hahn polynomials (cf.~\cite[Section 1.6]{KS1998R}):
\begin{align}
	\alpha(N+1)(\alpha &+\beta+2r) Q_r(\xi;\alpha-1,\beta,N+1) \label{backward shift operator for dual Hahn} \\
	=&\, (\alpha+r)(\alpha+\beta+r)(N+1-r)Q_r(\xi-1;\alpha,\beta,N) \notag \\
	& \quad -r(\alpha+\beta+N+1+r)(\beta+r)Q_{r-1}(\xi-1;\alpha,\beta,N). \notag
\end{align}
This can be routinely verified by writing the LHS as a linear combination of the polynomials $(1-\xi)_i$ $(0\leqslant i\leqslant r)$ using
\begin{equation*}
	(-\xi)_i=(1-\xi)_i-i(1-\xi)_{i-1},
\end{equation*}
and then comparing the coefficients of both sides.
Setting $\alpha=\ell-n$, $\beta=-\ell-1$, and $N=\ell-1$ in \eqref{backward shift operator for dual Hahn}, it follows that the first term of the RHS in \eqref{Wilson l} is rewritten as follows:
\begin{align*}
	\sum_{r=0}^{e-1} & \left(\!\binom{n}{r}\!-\!\binom{n}{r-1}\!\right)\hypergeometricseries{3}{2}{-\xi,-r,r-n-1}{\ell-n,-\ell}{1} \\
	&= \sum_{r=0}^{e-1} \frac{n!(n-2r+1)}{r!(n-r+1)!}Q_r(\xi;\alpha-1,\beta,N+1) \\
	 &= \frac{n!}{\ell(n-\ell)}\sum_{r=0}^{e-1}\left(\frac{(\ell-n+r)(r-n-1)(\ell-r)}{r!(n-r+1)!}Q_r(\xi-1;\alpha,\beta,N)\right. \\
	 & \qquad\qquad\qquad\qquad \left. -\frac{r(r+\ell-n-1)(r-\ell-1)}{r!(n-r+1)!}Q_{r-1}(\xi-1;\alpha,\beta,N) \right) \\
	 &= \frac{n!}{\ell(n-\ell)}\cdot(-1)\frac{(\ell-n+e-1)(\ell-e+1)}{(e-1)!(n-e+1)!}Q_{e-1}(\xi-1;\alpha,\beta,N) \\
	 &= \binom{n}{e-1}\frac{(n-\ell-e+1)(\ell-e+1)}{\ell(n-\ell)} \hypergeometricseries{3}{2}{1-\xi,1-e,e-n-1}{\ell-n+1, 1-\ell}{1}.
\end{align*}
Likewise, the first term of the RHS in \eqref{Wilson m} is given by
\begin{align*}
	\sum_{r=0}^{e-1} & \left(\!\binom{n}{r}\!-\!\binom{n}{r-1}\!\right)\hypergeometricseries{3}{2}{-\xi,-r,r-n-1}{m-n,-m}{1} \\
	&= \binom{n}{e-1}\frac{(n-m-e+1)(m-e+1)}{m(n-m)} \hypergeometricseries{3}{2}{1-\xi,1-e,e-n-1}{m-n+1, 1-m}{1}.
\end{align*}

\section{Zeros of the Hahn and Hermite polynomials}\label{sec: zeros}

Recall the Hahn polynomials $Q_r(\xi;\alpha,\beta,N)$ from \eqref{Hahn}.
Recall also that the zeros of orthogonal polynomials are always real and simple; see, e.g., \cite[Theorem 3.3.1]{Szego1975B}.
It is well known that we can obtain the Hermite polynomials as limits of the Hahn polynomials; cf.~\cite{KLS2010B,KS1998R}.
In this section, we revisit this limit process and describe the limit behavior of the zeros of the $Q_r(\xi;\alpha,\beta,N)$, in a special case which is suited to our purpose.

\begin{assump}\label{limit setting}
Throughout this section, we assume that $\alpha<-N$ and $\beta<-N$, so that the $Q_r(\xi;\alpha,\beta,N)$ satisfy the orthogonality relation \eqref{orthogonality}.
We consider the following limit:
\begin{equation*}
	\epsilon:=-\frac{\alpha+\beta}{\sqrt{\alpha\beta N}}\rightarrow +0.
\end{equation*}
We write
\begin{equation*}
	\alpha=\frac{\alpha_{\epsilon}}{\epsilon^2}, \qquad \beta=\frac{\beta_{\epsilon}}{\epsilon^2}, \qquad N=\frac{N_{\epsilon}}{\epsilon^2},
\end{equation*}
and assume further that
\begin{equation*}
	\lim_{\epsilon\rightarrow +0}\frac{N_{\epsilon}}{\alpha_{\epsilon}+\beta_{\epsilon}}=0, \qquad \lim_{\epsilon\rightarrow +0}\frac{\beta_{\epsilon}}{\alpha_{\epsilon}+\beta_{\epsilon}}=\rho\in[0,1].
\end{equation*}
\end{assump}

\begin{rem}
We do not require in Assumption \ref{limit setting} that $\alpha_{\epsilon},\beta_{\epsilon}$, and $N_{\epsilon}$ are uniquely determined by $\epsilon$.
In other words, these are multi-valued functions of $\epsilon$ in general (for admissible values of $\epsilon$), but their limit behaviors are uniformly governed by $\epsilon$.
\end{rem}

With reference to Assumption \ref{limit setting}, observe that
\begin{equation*}
	\lim_{\epsilon\rightarrow +0}\alpha_{\epsilon}=\lim_{\epsilon\rightarrow +0}\alpha \epsilon^2=\lim_{\epsilon\rightarrow +0}\frac{\alpha_{\epsilon}+\beta_{\epsilon}}{\beta_{\epsilon}}\cdot\frac{\alpha_{\epsilon}+\beta_{\epsilon}}{N_{\epsilon}}=-\infty.
\end{equation*}
Likewise, we have
\begin{equation*}
	\lim_{\epsilon\rightarrow +0}\beta_{\epsilon}=-\infty, \qquad \lim_{\epsilon\rightarrow +0}N_{\epsilon}=\frac{1}{\rho(1-\rho)} \in[4,\infty].
\end{equation*}

We will work with the \emph{normalized} (or \emph{monic}) Hahn polynomials:
\begin{equation}\label{monic Hahn}
	q_r(\xi)=q_r(\xi;\epsilon)=\frac{(\alpha+1)_r(-N)_r}{(r+\alpha+\beta+1)_r}Q_r(\xi;\alpha,\beta,N).
\end{equation}
Their recurrence relation is given by (cf.~\cite[Section 1.5]{KS1998R})
\begin{equation}\label{Hahn recurrence}
	\xi q_r(\xi)=q_{r+1}(\xi)+(a_r+b_r)q_r(\xi)+a_{r-1}b_rq_{r-1}(\xi),
\end{equation}
where $q_{-1}(\xi):=0$, and
\begin{align*}
	a_r&=\frac{(r+\alpha+\beta+1)(r+\alpha+1)(N-r)}{(2r+\alpha+\beta+1)(2r+\alpha+\beta+2)}, \\
	b_r&=\frac{r(r+\alpha+\beta+N+1)(r+\beta)}{(2r+\alpha+\beta)(2r+\alpha+\beta+1)}.
\end{align*}

For convenience, let
\begin{equation*}
	\lambda_{\epsilon}=\sqrt{\frac{2(\alpha_{\epsilon}+\beta_{\epsilon}+N_{\epsilon})}{\alpha_{\epsilon}+\beta_{\epsilon}}}.
\end{equation*}
Note that
\begin{equation}\label{lambda_zero}
	\lim_{\epsilon\rightarrow +0}\lambda_{\epsilon}=\sqrt{2}.
\end{equation}
Consider the polynomial $\tilde{q}_r(\eta;\epsilon)$ in the new indeterminate $\eta$ defined by
\begin{equation*}
	\tilde{q}_r(\eta)=\tilde{q}_r(\eta;\epsilon)=q_r\!\left(\frac{\lambda_{\epsilon}\eta}{\epsilon}+\frac{\alpha_{\epsilon}N_{\epsilon}}{(\alpha_{\epsilon}+\beta_{\epsilon})\epsilon^2}\right)\!\cdot\frac{\epsilon^r}{(\lambda_{\epsilon})^r} \in \mathbb{R}[\eta].
\end{equation*}
Note that $\tilde{q}_r(\eta)$ is also monic with degree $r$ in $\eta$.
Then \eqref{Hahn recurrence} becomes
\begin{equation}\label{modified recurrence}
	\eta\tilde{q}_r(\eta)=\tilde{q}_{r+1}(\eta)+\frac{1}{\lambda_{\epsilon}}\!\left((a_r+b_r)\epsilon-\frac{\alpha_{\epsilon}N_{\epsilon}}{(\alpha_{\epsilon}+\beta_{\epsilon})\epsilon}\right)\!\tilde{q}_r(\eta)+\frac{a_{r-1}b_r\epsilon^2}{(\lambda_{\epsilon})^2}\tilde{q}_{r-1}(\eta).
\end{equation}
It is a straightforward matter to show that
\begin{gather}
	\frac{1}{\lambda_{\epsilon}}\!\left((a_r+b_r)\epsilon-\frac{\alpha_{\epsilon}N_{\epsilon}}{(\alpha_{\epsilon}+\beta_{\epsilon})\epsilon}\right) = -(\mu_{\epsilon}+r\sigma_{\epsilon})\epsilon +O(\epsilon^3), \\
	\frac{a_{r-1}b_r\epsilon^2}{(\lambda_{\epsilon})^2} = \frac{r}{2} +O(\epsilon^2),
\end{gather}
where
\begin{equation*}
	\mu_{\epsilon}:=\frac{(\alpha_{\epsilon}-\beta_{\epsilon})N_{\epsilon}}{\lambda_{\epsilon} (\alpha_{\epsilon}+\beta_{\epsilon})^2}, \qquad \sigma_{\epsilon}:=\frac{(\alpha_{\epsilon}-\beta_{\epsilon})(\alpha_{\epsilon}+\beta_{\epsilon}+2N_{\epsilon})}{\lambda_{\epsilon} (\alpha_{\epsilon}+\beta_{\epsilon})^2}
\end{equation*}
are convergent:
\begin{equation}
	\lim_{\epsilon\rightarrow +0}\mu_{\epsilon}=0, \qquad \lim_{\epsilon\rightarrow +0}\sigma_{\epsilon}=\frac{1-2\rho}{\sqrt{2}}.
\end{equation}

Recall the \emph{Hermite polynomials} \cite[Section 1.13]{KS1998R}
\begin{equation*}
	H_r(\eta)=(2\eta)^r \hypergeometricseries{2}{0}{-r/2,-(r-1)/2}{-}{-\frac{1}{\eta^2}}\in\mathbb{R}[\eta] \quad (r=0,1,2,\ldots).
\end{equation*}
Their normalized recurrence relation is given by
\begin{equation}\label{Hermite recurrence}
	\eta h_r(\eta)=h_{r+1}(\eta)+\frac{r}{2}h_{r-1}(\eta),
\end{equation}
where
\begin{equation}\label{monic Hermite}
	h_r(\eta)=\frac{H_r(\eta)}{2^r},
\end{equation}
and $h_{-1}(\eta):=0$.
We also note that
\begin{equation}\label{derivative of h}
	\frac{d h_r}{d\eta}(\eta)=rh_{r-1}(\eta),
\end{equation}
and that
\begin{equation}\label{Hermite is symmetric}
	h_r(-\eta)=(-1)^rh_r(\eta).
\end{equation}
Since $\tilde{q}_0(\eta)=h_0(\eta)=1$, it follows from \eqref{modified recurrence}--\eqref{Hermite recurrence} that
\begin{equation}\label{constant term}
	\lim_{\epsilon\rightarrow+0}\tilde{q}_r(\eta;\epsilon)=h_r(\eta)
\end{equation}
in the sense of coefficient-wise convergence.

We now set
\begin{equation*}
	\tilde{q}_r(\eta;0)=h_r(\eta),
\end{equation*}
and discuss partial derivatives of $\tilde{q}_r(\eta;\epsilon)$ as a bivariate function of $\eta$ and $\epsilon$.
First, it follows from \eqref{derivative of h} and \eqref{constant term} that
\begin{equation}\label{partial derivative wrt eta}
	\lim_{\epsilon\rightarrow+0}\frac{\partial\tilde{q}_r}{\partial\eta}(\eta; \epsilon)=\frac{dh_r}{d\eta}(\eta)=rh_{r-1}(\eta).
\end{equation}
Concerning the partial differentiability of $\tilde{q}_r(\eta;\epsilon)$ with respect to $\epsilon$, it follows that

\begin{lem}\label{partial derivative wrt epsilon}
The function $\tilde{q}_r(\eta;\epsilon)$ is partially right differentiable with respect to $\epsilon$ at $(\eta,0)$, and we have
\begin{equation*}
	\frac{\partial\tilde{q}_r}{\partial\epsilon}(\eta;0) =  \frac{r(1-2\rho)}{3\sqrt{2}} \!\left( (r-1+\eta^2) h_{r-1}(\eta) -\eta h_r(\eta) \right).
\end{equation*}
\end{lem}

\begin{proof}
Throughout the proof, we fix $\eta\in\mathbb{R}$ and set
\begin{equation*}
	\Delta_r(\epsilon)=\Delta_r(\eta;\epsilon)=\frac{\tilde{q}_r(\eta;\epsilon)-h_r(\eta)}{\epsilon}.
\end{equation*}
It follows from \eqref{modified recurrence}--\eqref{Hermite recurrence} and \eqref{constant term} that
\begin{align}
	\eta \Delta_r(\epsilon) &= \Delta_{r+1}(\epsilon) -(\mu_{\epsilon}+r\sigma_{\epsilon})\tilde{q}_r(\eta;\epsilon) +\frac{r}{2} \Delta_{r-1}(\epsilon) +O(\epsilon) \label{degree 1 recurrence} \\
	&= \Delta_{r+1}(\epsilon) -r\sigma_0 h_r(\eta) +\frac{r}{2} \Delta_{r-1}(\epsilon) +o(1), \notag
\end{align}
where we set
\begin{equation*}
	\sigma_0:=\lim_{\epsilon\rightarrow +0}\sigma_{\epsilon}=\frac{1-2\rho}{\sqrt{2}}
\end{equation*}
for brevity.
Since $\tilde{q}_0(\eta;\epsilon)=1$, we have $\Delta_0(\epsilon)=0$.
Solving the recurrence \eqref{degree 1 recurrence} using this initial condition and \eqref{Hermite recurrence}, we routinely obtain
\begin{equation*}
	\Delta_r(\epsilon)=\frac{r(r-1)}{2}\sigma_0 h_{r-1}(\eta) +\frac{r(r-1)(r-2)}{12}\sigma_0 h_{r-3}(\eta)+o(1),
\end{equation*}
where $h_{-1}(\eta)=h_{-2}(\eta)=h_{-3}(\eta):=0$.
It follows that $\tilde{q}_r(\eta;\epsilon)$ is partially right differentiable with respect to $\epsilon$ at $(\eta,0)$:
\begin{align*}
	\frac{\partial\tilde{q}_r}{\partial\epsilon}(\eta;0) &= \lim_{\epsilon\rightarrow+0} \Delta_r(\epsilon) \\
	&= \frac{r(r-1)}{2}\sigma_0 h_{r-1}(\eta) +\frac{r(r-1)(r-2)}{12}\sigma_0 h_{r-3}(\eta).
\end{align*}
Finally, from \eqref{Hermite recurrence} it follows that
\begin{align*}
	\frac{\partial\tilde{q}_r}{\partial\epsilon}(\eta;0) &= \frac{r(r-1)}{2}\sigma_0 h_{r-1}(\eta) +\frac{r(r-1)}{6}\sigma_0 \bigl(\eta h_{r-2}(\eta)-h_{r-1}(\eta) \bigr) \\
	&= \frac{r(r-1)}{3}\sigma_0 h_{r-1}(\eta) +\frac{r}{3}\sigma_0\eta \bigl(\eta h_{r-1}(\eta)-h_r(\eta) \bigr) \\
	&= \frac{r\sigma_0}{3}\!\left((r-1+\eta^2) h_{r-1}(\eta) -\eta h_r(\eta)\right),
\end{align*}
as desired.
\end{proof}

\begin{prop}\label{behavior of zeros of Hahn}
Recall Assumption \ref{limit setting}.
Fix a positive integer $e$, and let
\begin{gather*}
	\xi_{-\lfloor e/2\rfloor}<\dots<\xi_{-1}<(\xi_0)<\xi_1<\dots<\xi_{\lfloor e/2 \rfloor}, \\
	\eta_{-\lfloor e/2\rfloor}<\dots<\eta_{-1}<(\eta_0)<\eta_1<\dots<\eta_{\lfloor e/2 \rfloor}
\end{gather*}
be the zeros of $q_e(\xi;\epsilon)$ and $h_e(\eta)$ from \eqref{monic Hahn} and \eqref{monic Hermite}, respectively, where $\xi_0$ and $\eta_0$ appear only when $e$ is odd.
Then $\xi_i$ satisfies
\begin{equation*}
	\lim_{\epsilon\rightarrow +0}\left( \xi_i -\frac{\lambda_{\epsilon}\eta_i}{\epsilon}-\frac{\alpha_{\epsilon}N_{\epsilon}}{(\alpha_{\epsilon}+\beta_{\epsilon})\epsilon^2}\right) =  \frac{2\rho-1}{3}\bigl(e-1+(\eta_i)^2\bigr)
\end{equation*}
as a function of $\epsilon$, for $i=-\lfloor e/2\rfloor,\dots,-1,(0),1,\dots,\lfloor e/2 \rfloor$.
\end{prop}

\begin{proof}
Define $\tau_i$ by
\begin{equation*}
	\xi_i=\frac{\lambda_{\epsilon}(\eta_i+\tau_i)}{\epsilon}+\frac{\alpha_{\epsilon}N_{\epsilon}}{(\alpha_{\epsilon}+\beta_{\epsilon})\epsilon^2},
\end{equation*}
so that $\eta_i+\tau_i$ is a zero of $\tilde{q}_e(\eta;\epsilon)$.
Then, from \eqref{constant term} it follows that
\begin{equation}\label{difference approach 0}
	\lim_{\epsilon\rightarrow +0} \tau_i=0.
\end{equation}
For the moment, fix $i$.
Then we have
\begin{equation*}
	0 = \tilde{q}_e(\eta_i+\tau_i; \epsilon) = \tilde{q}_e(\eta_i; \epsilon) + \frac{\partial\tilde{q}_e}{\partial\eta}(\eta_i+\theta\tau_i; \epsilon)\tau_i
\end{equation*}
for some $\theta\in(0,1)$ depending on $\epsilon$.
Hence, from \eqref{partial derivative wrt eta}, \eqref{difference approach 0}, Lemma \ref{partial derivative wrt epsilon},
and since
\begin{equation*}
	\tilde{q}_e(\eta_i; 0)=h_e(\eta_i)=0,
\end{equation*}
it follows that
\begin{align*}
	\lim_{\epsilon\rightarrow +0} \frac{\tau_i}{\epsilon} &= -\frac{1}{eh_{e-1}(\eta_i)} \lim_{\epsilon\rightarrow +0} \frac{\tilde{q}_e(\eta_i; \epsilon)}{\epsilon} \\
	&= -\frac{1}{eh_{e-1}(\eta_i)} \frac{\partial\tilde{q}_e}{\partial\epsilon}(\eta_i;0) \\
	&= \frac{2\rho-1}{3\sqrt{2}} \bigl(e-1+(\eta_i)^2\bigr),
\end{align*}
where we note that $h_e(\eta)$ and $h_{e-1}(\eta)$ have no common zero by the general theory of orthogonal polynomials; see, e.g., \cite[Theorem 3.3.2]{Szego1975B}.
By \eqref{lambda_zero}, we have
\begin{equation*}
	\lim_{\epsilon\rightarrow +0}\left( \xi_i -\frac{\lambda_{\epsilon}\eta_i}{\epsilon}-\frac{\alpha_{\epsilon}N_{\epsilon}}{(\alpha_{\epsilon}+\beta_{\epsilon})\epsilon^2}\right) = \lim_{\epsilon\rightarrow +0} \frac{\lambda_{\epsilon}\tau_i}{\epsilon} =  \frac{2\rho-1}{3} \bigl(e-1+(\eta_i)^2\bigr).
\end{equation*}
This completes the proof.
\end{proof}

The following is part of the estimates on the zeros of $h_e(\eta)$ used in \cite{Bannai1977QJMO}.\footnote{Bannai \cite{Bannai1977QJMO} worked with the polynomial $\sqrt{2^e}h_e(\eta/\sqrt{2})$. We may remark that the upper bounds $\sqrt{3}$ mentioned in Proposition 13\,(i) and (ii) in \cite{Bannai1977QJMO} should both be $3$. See also \cite[Proposition 2.4]{DS-G2013JAC}.}

\begin{prop}[{\cite[Proposition 13]{Bannai1977QJMO}}]\label{bounds on zeros of Hermite}
Fix a positive integer $e$, and let the $\eta_i$ be as in Proposition \ref{behavior of zeros of Hahn}.
Then $\eta_{-i}=-\eta_i$ for all $i$.
Moreover, the following hold:
\begin{enumerate}
\item If $e$ is odd and $e\geqslant 5$, then $\eta_0=0$ and $(\eta_1)^2<3/2$.
\item If $e$ is even and $e\geqslant 8$, then $(\eta_2)^2-(\eta_1)^2<3/2$.
\end{enumerate}
\end{prop}

\begin{proof}
That $\eta_{-i}=-\eta_i$ is immediate from \eqref{Hermite is symmetric}.
We now write $\eta_i=\eta_i^e$ to compare these zeros for different values of $e$.
Then, as an application of Sturm's method, it follows that
\begin{equation*}
	\sqrt{2e+1}\,\eta_i^e<\sqrt{2e'+1}\,\eta_i^{e'} \qquad (i=1,2,\dots,\lfloor e'/2\rfloor),
\end{equation*}
whenever $e'<e$ and $e'\equiv e\ (\operatorname{mod}\, 2)$; see the comments preceding (6.31.19) in \cite{Szego1975B}.
Since
\begin{equation*}
	h_3(\eta)=\eta^3-\frac{3}{2}\eta, \qquad h_4(\eta)=\eta^4-3\eta^2+\frac{3}{4},
\end{equation*}
we have
\begin{equation*}
	\eta_1^3=\sqrt{\frac{3}{2}}, \qquad \eta_2^4=\sqrt{\frac{3+\sqrt{6}}{2}}.
\end{equation*}
Hence, for odd $e\geqslant 5$ we have
\begin{equation*}
	(\eta_1^e)^2<\frac{7}{2e+1}(\eta_1^3)^2=\frac{21}{4e+2}<\frac{3}{2},
\end{equation*}
and for even $e\geqslant 8$ we have
\begin{equation*}
	(\eta_2^e)^2-(\eta_1^e)^2<(\eta_2^e)^2<\frac{9}{2e+1} (\eta_2^4)^2=\frac{27+9\sqrt{6}}{4e+2}<\frac{3}{2},
\end{equation*}
as desired.
\end{proof}

\section{A finiteness result for tight relative \texorpdfstring{$2e$}{2e}-designs on two shells in \texorpdfstring{$\mathcal{Q}_n$}{Qn}}\label{sec: finiteness result}

In this section, we prove that

\begin{thm}\label{counterpart of Bannai's theorem}
For any $\delta\in(0,1/2)$, there exists $e_0=e_0(\delta)>0$ with the property that, for every given integer $e\geqslant e_0$ and each constant $c>0$, there are only finitely many tight relative $2e$-designs $(Y,\omega)$ (up to scalar multiples of $\omega$) supported on two shells $X_{\ell}\sqcup X_m$ in $\mathcal{Q}_n$ satisfying Assumption \ref{assumption on Y} such that 
\begin{equation}\label{l/n tends to 0 fast}
	\ell<c\cdot n^{\delta}.
\end{equation}
\end{thm}

Our proof is an application of Bannai's method from \cite{Bannai1977QJMO}.
We will use the following result, which is a variation of \cite[Satz I]{Schur1929SPAW}:

\begin{prop}\label{variation of Schur's theorem}
For any $\vartheta>0$ and $\delta\in(0,1/\vartheta)$, there exists $\mathfrak{k}_0=\mathfrak{k}_0(\vartheta,\delta)>0$ such that the following holds for every given integer $\mathfrak{k}\geqslant \mathfrak{k}_0$ and each constant $c>0$:
for all but finitely many pairs $(\mathfrak{a},\mathfrak{b})$ of positive integers with
\begin{equation*}
	\mathfrak{b}<c\cdot\mathfrak{a}^{\delta},
\end{equation*}
the product of $\mathfrak{k}$ consecutive odd integers
\begin{equation*}
	(2\mathfrak{a}+1)(2\mathfrak{a}+3)\cdots(2\mathfrak{a}+2\mathfrak{k}-1)
\end{equation*}
has a prime factor which is greater than $2\mathfrak{k}+1$ and whose exponent in this product is greater than that in
\begin{equation*}
	(\mathfrak{b}+1)(\mathfrak{b}+2)\cdots(\mathfrak{b}+\lfloor \vartheta\mathfrak{k}\rfloor).
\end{equation*}
\end{prop}

\noindent
The proof of Proposition \ref{variation of Schur's theorem} will be deferred to the appendix.

We will establish Theorem \ref{counterpart of Bannai's theorem} by contradiction:

\begin{assump}\label{prove by contradiction}
We fix $\delta\in (0,1/2)$.
Let  $\mathfrak{k}_0=\mathfrak{k}_0(2,\delta)>0$ be as in Proposition \ref{variation of Schur's theorem} (applied to $\vartheta=2$), and set
\begin{equation*}
	e_0=e_0(\delta)=\max\{2\mathfrak{k}_0,8\}.
\end{equation*}
We also fix a positive integer $e\geqslant e_0$ and a constant $c>0$.
Throughout the proof, we assume that there exist infinitely many tight relative $2e$-designs $(Y,\omega)$ in question.
\end{assump}

Let $\Theta$ denote the set of triples $(\ell,m,n)\in\mathbb{N}^3$ taken by those $(Y,\omega)$ in Assumption \ref{prove by contradiction}.
Recall from Proposition \ref{ratio of weights} that $\omega$ is uniquely determined by $Y$ up to a scalar multiple.
Moreover, for each $(\ell,m,n)\in\Theta$ there are only finitely many choices for $Y$.
Hence we have
\begin{equation}\label{Theta is infinite}
	|\Theta|=\infty.
\end{equation}

For the moment, we fix $(\ell,m,n)\in\Theta$ and consider the polynomial $\psi_e^{\ell,m}(\xi)$ (which also depends on $n$) from \eqref{Wilson lm}.
We recall that
\begin{equation*}
	\psi_e^{\ell,m}(\xi)=Q_e(\xi;\alpha,\beta,N),
\end{equation*}
where
\begin{equation}\label{Hahn parameters}
	\alpha=m-n-1, \qquad \beta=-m-1, \qquad N=\ell.
\end{equation}
We note that $\alpha,\beta<-N$ in view of Assumption \ref{assumption on Y}.
By Theorem \ref{counterpart of Wilson's theorem}\,(ii), if we let
\begin{equation}\label{zeros of Wilson polynomial}
	\xi_{-\lfloor e/2\rfloor}<\dots<\xi_{-1}<(\xi_0)<\xi_1<\dots<\xi_{\lfloor e/2 \rfloor}
\end{equation}
denote the zeros of $\psi_e^{\ell,m}(\xi)$ (cf.~Proposition \ref{behavior of zeros of Hahn}), then we have
\begin{equation}\label{locations of zeros}
	\xi_i\in\{0,1,\dots,\ell\} \qquad \text{for all} \ i.
\end{equation}
We also rewrite $\psi_e^{\ell,m}(\xi)$ as follows:
\begin{equation*}
	\psi_e^{\ell,m}(\xi)=\sum_{i=0}^e s_{e-i} (-1)^i(-\xi)_i,
\end{equation*}
where
\begin{equation*}
	s_{e-i}=\binom{e}{i}\frac{(e-n-1)_i}{(m-n)_i(-\ell)_i} \qquad (0\leqslant i\leqslant e).
\end{equation*}
From \eqref{locations of zeros} it follows that the polynomial $\psi_e^{\ell,m}(\xi)/s_0$ is monic and integral:
\begin{equation}\label{monic integral}
	\frac{\psi_e^{\ell,m}(\xi)}{s_0}=\sum_{i=0}^e \frac{s_{e-i}}{s_0} (-1)^i(-\xi)_i=(\xi-\xi_{-\lfloor e/2\rfloor})\cdots(\xi-\xi_{\lfloor e/2\rfloor})\in\mathbb{Z}[\xi],
\end{equation}
where the factor $(\xi-\xi_0)$ appears only when $e$ is odd.
Since $(-1)^i(-\xi)_i$ is also monic and integral, and has degree $i$ for $0\leqslant i\leqslant e$, it follows that
\begin{equation}\label{coefficients}
	\frac{s_i}{s_0}=(-1)^i\binom{e}{i} \frac{(n-m-e+1)_i(\ell-e+1)_i}{(n-2e+2)_i}\in\mathbb{Z}\backslash\{0\} \qquad (0\leqslant i\leqslant e),
\end{equation}
where that these coefficients are non-zero follows from Assumption \ref{assumption on Y}.

We now consider the map $f:\Theta\rightarrow [0,1]^2$ defined by
\begin{equation*}
	f(\ell,m,n)=\left(\frac{\ell}{n},\frac{m}{n}\right)\in [0,1]^2 \qquad ((\ell,m,n)\in\Theta).
\end{equation*}
Recall \eqref{Theta is infinite}.
Moreover, from \eqref{l/n tends to 0 fast} it follows that
\begin{equation}\label{inverse is finite}
	|f^{-1}(a,b)|<\infty \qquad ((a,b)\in[0,1]^2).
\end{equation}
Hence it follows that
\begin{equation*}
	|f(\Theta)|=\infty,
\end{equation*}
so that $f(\Theta)$ has at least one accumulation point in $[0,1]^2$.
Again by \eqref{l/n tends to 0 fast}, such an accumulation point must be of the form
\begin{equation*}
	(0,\rho) \in [0,1]^2.
\end{equation*}

We next show that the parameters $\alpha,\beta$, and $N$ from \eqref{Hahn parameters} satisfy Assumption \ref{limit setting} when $f(\ell,m,n)\rightarrow(0,\rho)$.

\begin{claim}\label{claim 1}
$\ell,m,n-m\rightarrow\infty$ as $f(\ell,m,n)\rightarrow(0,\rho)$.
\end{claim}

\begin{proof}
Since $m,n-m\geqslant\ell$ by Assumption \ref{assumption on Y}, it suffices to show that $\ell\rightarrow\infty$.
Suppose the contrary, i.e., that there is a sequence $(\ell_k,m_k,n_k)$ $(k\in\mathbb{N})$ of distinct elements of $\Theta$ such that
\begin{equation*}
	\lim_{k\rightarrow\infty}f(\ell_k,m_k,n_k)=(0,\rho), \qquad \sup_k \ell_k<\infty.
\end{equation*}
Since the $\ell_k$ are bounded, it follows from \eqref{locations of zeros} and \eqref{monic integral} that there are only finitely many choices for $\psi_e^{\ell,m}(\xi)/s_0$ when $(\ell,m,n)$ ranges over this sequence.
In particular, there are only finitely many choices for each of the coefficients $s_1/s_0$ and $s_2/s_0$, and hence the same is true (cf.~\eqref{coefficients}) for each of
\begin{equation*}
	\frac{n-m-e+1}{n-2e+2}, \qquad \frac{n-m-e+2}{n-2e+3}.
\end{equation*}
However, it is immediate to see that these distinct scalars in turn determine $n$ and $m$ uniquely, from which it follows that the $n_k$ are bounded, a contradiction.
\end{proof}

\begin{claim}\label{claim 2}
$\ell m(n-m)/n^2\rightarrow\infty$ as $f(\ell,m,n)\rightarrow(0,\rho)$.
\end{claim}

\begin{proof}
If $0<\rho<1$ then the result follows from Claim \ref{claim 1} and since
\begin{equation*}
	\frac{m(n-m)}{n^2}\rightarrow \rho(1-\rho)>0.
\end{equation*}

Suppose next that $\rho=1$.
Suppose moreover that there is a sequence $(\ell_k,m_k,n_k)$ $(k\in\mathbb{N})$ of distinct elements of $\Theta$ such that
\begin{equation*}
	\lim_{k\rightarrow\infty}f(\ell_k,m_k,n_k)=(0,1), \qquad \sup_k \frac{\ell_km_k(n_k-m_k)}{(n_k)^2}<\infty.
\end{equation*}
Since $m_k/n_k\rightarrow 1$, we then have
\begin{equation*}
	\sup_k \frac{\ell_k(n_k-m_k)}{n_k}<\infty.
\end{equation*}
Let
\begin{equation*}
	r_k=\frac{(n_k-m_k-e+1)(\ell_k-e+1)}{n_k-2e+2}, \qquad t_k=\frac{(n_k-m_k-e+2)(\ell_k-e+2)}{n_k-2e+3}.
\end{equation*}
Then the $r_k$ and the $t_k$ are bounded since
\begin{equation*}
	r_k\approx t_k\approx \frac{\ell_k(n_k-m_k)}{n_k}
\end{equation*}
by Claim \ref{claim 1}.
From \eqref{coefficients} it follows that $s_1/s_0$ and $s_2/s_0$ are bounded as well, and hence take only finitely many non-zero integral values when $(\ell,m,n)$ ranges over this sequence.
It follows that the $r_k$ and the $t_k$ can assume only finitely many values, and then since $r_k\approx t_k$ we must have $r_k=t_k$ for sufficiently large $k$.
However, it is again immediate to see that $r_k\ne t_k$ for every $k\in\mathbb{N}$, and hence this is absurd.
It follows that the result holds when $\rho=1$.

Finally, suppose that $\rho=0$.
For every $(\ell,m,n)\in\Theta$ we have
\begin{gather*}
	e\frac{(m-e+1)(\ell-e+1)}{n-2e+2}=\frac{s_1}{s_0}+e(\ell-e+1), \\
	\binom{e}{2}\frac{(m-e+1)_2(\ell-e+1)_2}{(n-2e+2)_2}=\frac{s_2}{s_0}+(e-1)(\ell-e+2)\frac{s_1}{s_0}+\binom{e}{2}(\ell-e+1)_2.
\end{gather*}
From \eqref{coefficients} and Assumption \ref{assumption on Y} it follows that these scalars are non-zero integers.
By the same argument as above, but working with these two scalars instead of $s_1/s_0$ and $s_2/s_0$, we conclude that the result holds in this case as well.
\end{proof}

By Claims \ref{claim 1} and \ref{claim 2}, it follows that the parameters $\alpha,\beta$, and $N$ from \eqref{Hahn parameters} satisfy Assumption \ref{limit setting} when $f(\ell,m,n)\rightarrow(0,\rho)$, since
\begin{equation*}
	-\frac{\alpha+\beta}{\sqrt{\alpha\beta N}}\approx\frac{n}{\sqrt{\ell m(n-m)}}, \qquad \frac{N}{\alpha+\beta}\approx -\frac{\ell}{n}, \qquad \frac{\beta}{\alpha+\beta}\approx \frac{m}{n}.
\end{equation*}
Note that the scalar $\rho$ in Assumption \ref{limit setting} agrees with the one used here in this case.
Hence we are now in the position to apply the results of the previous section to $\psi_e^{\ell,m}(\xi)$, which is the Hahn polynomial having these parameters.

\begin{claim}\label{claim 3}
We have $\rho=1/2$.
In particular, $(0,1/2)$ is a unique accumulation point of $f(\Theta)$.
Moreover, we have $n=2m$ for all but finitely many $(\ell,m,n)\in\Theta$.
\end{claim}

\begin{proof}
Let the $\xi_i$ be as in \eqref{zeros of Wilson polynomial}.
Then from Propositions \ref{behavior of zeros of Hahn} and \ref{bounds on zeros of Hermite} it follows that
\begin{equation}\label{show symmetry}
	\xi_i+\xi_{-i}-\xi_j-\xi_{-j}\rightarrow \frac{4\rho-2}{3}\bigl((\eta_i)^2-(\eta_j)^2\bigr) \qquad \text{for all} \ i,j,
\end{equation}
as $f(\ell,m,n)\rightarrow (0,\rho)$, where the $\eta_i$ are the zeros of the monic Hermite polynomial $h_e(\eta)$ from \eqref{monic Hermite} as in Proposition \ref{behavior of zeros of Hahn}.
Recall that $e\geqslant 8$ by Assumption \ref{prove by contradiction}.
Set $(i,j)=(1,0)$ in \eqref{show symmetry} if $e$ is odd, and $(i,j)=(2,1)$ if $e$ is even.
Then, since
\begin{equation*}
	\left|\frac{4\rho-2}{3}\right|\leqslant \frac{2}{3},
\end{equation*}
it follows from Proposition \ref{bounds on zeros of Hermite} that the RHS in \eqref{show symmetry} lies in the open interval $(-1,1)$.
However, the LHS in \eqref{show symmetry} is always an integer by \eqref{locations of zeros}, so that this is possible only when the RHS equals zero, i.e., $\rho=1/2$.
In particular, we have shown that $(0,1/2)$ is a unique accumulation point of $f(\Theta)$.

Again by \eqref{locations of zeros} and \eqref{show symmetry}, we then have
\begin{equation*}
	\xi_i+\xi_{-i}=\xi_j+\xi_{-j} \qquad \text{for all} \ i,j,
\end{equation*}
provided that $f(\ell,m,n)$ is sufficiently close to $(0,1/2)$.
By the uniqueness of the accumulation point and \eqref{inverse is finite}, this last condition on $f(\ell,m,n)$ can be rephrased as ``for all but finitely many $(\ell,m,n)\in\Theta$.''
Now, let $\tilde{\xi}$ be the average of the zeros $\xi_i$ of $\psi_e^{\ell,m}(\xi)$.
Then the above identity means that the $\xi_i$ are symmetric with respect to $\tilde{\xi}$.
Hence, if we write
\begin{equation*}
	\frac{\psi_e^{\ell,m}(\xi)}{s_0}=\sum_{i=0}^e w_{e-i}(\xi-\tilde{\xi})^i,
\end{equation*}
then we have
\begin{equation*}
	w_{2i-1}=0 \qquad (1\leqslant i\leqslant \lceil e/2\rceil)
\end{equation*}
for all but finitely many $(\ell,m,n)\in\Theta$.
On the other hand, using \eqref{monic integral} and \eqref{coefficients}, we routinely obtain
\begin{align*}
	w_3&=\binom{e}{3}(n-2\ell)(n-2 m) \\
	& \qquad \times \frac{(\ell-e+1) (m-e+1) (n-\ell-e+1) (n-m-e+1)}{(n-2 e+2)^3(n-2 e+3)(n-2 e+4)}.
\end{align*}
Hence, by Assumption \ref{assumption on Y}, that $w_3=0$ forces $n=2m$.
The claim is proved.
\end{proof}

By virtue of Claim \ref{claim 3}, we may now assume without loss of generality that
\begin{equation*}
	n=2m \qquad ((\ell,m,n)\in\Theta),
\end{equation*}
by discarding a finite number of exceptions.
Set
\begin{equation*}
	\mathfrak{k}=\left\lfloor \frac{e}{2}\right\rfloor,
\end{equation*}
and let $c'$ be a constant such that $c'>2^{\delta}c$.
Note that
\begin{equation*}
	\mathfrak{k}\geqslant\mathfrak{k}_0=\mathfrak{k}_0(2,\delta)
\end{equation*}
by Assumption \ref{prove by contradiction}.
Let $(\ell,m,2m)\in\Theta$.
We have
\begin{equation*}
	c\cdot (2m)^{\delta}<c'\cdot (m-e+1)^{\delta}
\end{equation*}
provided that $m$ is large.
Hence it follows from Proposition \ref{variation of Schur's theorem} (applied to $\vartheta=2$) and \eqref{l/n tends to 0 fast} that if $m$ is sufficiently large then there is a prime $p>2\mathfrak{k}+1$ such that
\begin{equation*}
	\nu_p((2m-2e+3)(2m-2e+5)\cdots(2m-2e+2\mathfrak{k}+1))>\nu_p((\ell-e+1)_{2\mathfrak{k}}),
\end{equation*}
where $\nu_p(\mathfrak{n})$ denotes the exponent of $p$ in $\mathfrak{n}$.
Assuming that this is the case, let $i$ $(1\leqslant i\leqslant\mathfrak{k})$ be such that
\begin{equation*}
	\nu_p(2m-2e+2i+1)>0.
\end{equation*}
Observe that $i$ is unique since $p>2\mathfrak{k}+1$, so that we have
\begin{equation*}
	\nu_p(2m-2e+2i+1)>\nu_p((\ell-e+1)_{2\mathfrak{k}}).
\end{equation*}
Moreover, we have
\begin{equation*}
	\gcd(2m-2e+2i+1,m-e+i+j)=\gcd(2j-1,m-e+i+j)<p
\end{equation*}
for $1\leqslant j\leqslant i$, from which it follows that
\begin{equation*}
	\nu_p((m-e+i+1)_i)=0.
\end{equation*}
By these comments and since
\begin{equation*}
	2i\leqslant e<p,
\end{equation*}
it follows from \eqref{coefficients} (with $n=2m$) that
\begin{align*}
	\nu_p\!\left(\frac{s_{2i}}{s_0}\right)&=\nu_p\!\left(\frac{(m-e+1)_{2i}(\ell-e+1)_{2i}}{(2m-2e+2)_{2i}}\right) \\
	&= \nu_p\!\left(\frac{(m-e+i+1)_i(\ell-e+1)_{2i}}{2^i(2m-2e+3)(2m-2e+5)\cdots(2m-2e+2i+1)}\right) \\
	&<0.
\end{align*}
However, this contradicts the fact that $s_{2i}/s_0$ is a non-zero integer.
Hence we now conclude that $\Theta$ must be finite.

The proof of Theorem \ref{counterpart of Bannai's theorem} is complete.

\section*{Acknowledgments}

Hajime Tanaka was supported by JSPS KAKENHI Grant Numbers JP25400034 and JP17K05156.
Yan Zhu was supported by NSFC Grant No.~11801353.
This work was also partially supported by the Research Institute for Mathematical Sciences at Kyoto University.


\appendix

\section*{Appendix. Proof of Proposition \ref{variation of Schur's theorem}}

Our proof of Proposition \ref{variation of Schur's theorem} is a slight modification of (the first part of) that of \cite[Satz I]{Schur1929SPAW}.

For a positive integer $\mathfrak{n}$, let
\begin{equation*}
	\chi_{\mathfrak{n}}=1\cdot 3\cdot 5\cdots (2\mathfrak{n}-1)=\frac{(2\mathfrak{n})!}{2^{\mathfrak{n}}\mathfrak{n}!}.
\end{equation*}
Observe that the exponent $\nu_p(\chi_{\mathfrak{n}})$ of an odd prime $p$ in $\chi_{\mathfrak{n}}$ is given by
\begin{equation}\label{exponent of p}
	\nu_p(\chi_{\mathfrak{n}}) = \sum_{i=1}^{\lfloor\log_p (2\mathfrak{n})\rfloor} \!\!\left(\left\lfloor\frac{2\mathfrak{n}}{p^i}\right\rfloor - \left\lfloor\frac{\mathfrak{n}}{p^i}\right\rfloor \right) = \sum_{i=1}^{\lfloor\log_p (2\mathfrak{n})\rfloor} \! \left\lfloor\frac{\mathfrak{n}}{p^i}+\frac{1}{2}\right\rfloor,
\end{equation}
where we have used
\begin{equation*}
	\lfloor \xi\rfloor + \left\lfloor \xi+\frac{1}{2}\right\rfloor = \lfloor 2\xi\rfloor \qquad (\xi\in\mathbb{R}).
\end{equation*}

Now, let $(\mathfrak{a},\mathfrak{b})$ be a pair of positive integers with
\begin{equation}\label{b is bounded in terms of a}
	\mathfrak{b}<c\cdot\mathfrak{a}^{\delta},
\end{equation}
which does not satisfy the desired property about a prime factor; in other words,
\begin{equation}\label{inequality for exponents}
	\nu_p(\chi_{\mathfrak{a}+\mathfrak{k}}) - \nu_p(\chi_{\mathfrak{a}}) \leqslant \nu_p((\mathfrak{b}+\lfloor \vartheta\mathfrak{k}\rfloor)!) - \nu_p(\mathfrak{b}!) \qquad \text{if} \ p>2\mathfrak{k}+1.
\end{equation}
Our aim is to show that $\mathfrak{a}$ is bounded in terms of $\vartheta,\delta,c$, and $\mathfrak{k}$, and hence so is $\mathfrak{b}$ by \eqref{b is bounded in terms of a}, from which it follows that there are only finitely many such pairs.
(We will specify $\mathfrak{k}_0=\mathfrak{k}_0(\vartheta,\delta)$ at the end of the proof.)
To this end, we may assume for example that
\begin{equation}\label{a is large}
	\mathfrak{a}>\mathfrak{k}, \qquad c\cdot\mathfrak{a}^{\delta}>\mathfrak{k}+1.
\end{equation}
Without loss of generality, we may also assume that
\begin{equation}\label{b is large}
	\mathfrak{b}>\mathfrak{k},
\end{equation}
for otherwise the pair $(\mathfrak{a},\mathfrak{k}+1)$ would also satisfy \eqref{b is bounded in terms of a} and \eqref{inequality for exponents}.

Let
\begin{equation*}
	\mathfrak{s}=\frac{\chi_{\mathfrak{a}+\mathfrak{k}}}{\chi_{\mathfrak{k}} \chi_{\mathfrak{a}}}\cdot\frac{\mathfrak{b}!}{(\mathfrak{b}+\lfloor\vartheta\mathfrak{k}\rfloor)!}.
\end{equation*}
Then from \eqref{inequality for exponents} it follows that
\begin{equation}\label{upper bound on s}
	\mathfrak{s}\leqslant \prod_{3\leqslant p\leqslant 2\mathfrak{k}+1} p^{\nu_p(\mathfrak{s})},
\end{equation}
where the product in the RHS is over the odd primes $p\leqslant 2\mathfrak{k}+1$, and where
\begin{equation*}
	\nu_p(\mathfrak{s})= \nu_p(\chi_{\mathfrak{a}+\mathfrak{k}}) - \nu_p(\chi_{\mathfrak{k}}) -\nu_p(\chi_{\mathfrak{a}}) + \nu_p(\mathfrak{b}!) -\nu_p((\mathfrak{b}+\lfloor\vartheta\mathfrak{k}\rfloor)!).
\end{equation*}
By \eqref{exponent of p}, for every odd prime $p$ we have
\begin{align*}
	\nu_p(\mathfrak{s}) &\leqslant \nu_p(\chi_{\mathfrak{a}+\mathfrak{k}}) - \nu_p(\chi_{\mathfrak{k}}) -\nu_p(\chi_{\mathfrak{a}}) \\
	&= \sum_{i=1}^{\lfloor\log_p (2\mathfrak{a}+2\mathfrak{k})\rfloor} \!\! \left( \left\lfloor\frac{\mathfrak{a}+\mathfrak{k}}{p^i}+\frac{1}{2}\right\rfloor - \left\lfloor\frac{\mathfrak{k}}{p^i}+\frac{1}{2}\right\rfloor - \left\lfloor\frac{\mathfrak{a}}{p^i}+\frac{1}{2}\right\rfloor \right).
\end{align*}
Note that
\begin{equation*}
	\left\lfloor \xi+\eta+\frac{1}{2} \right\rfloor - \left\lfloor \xi+\frac{1}{2} \right\rfloor - \left\lfloor \eta+\frac{1}{2} \right\rfloor \in\{-1,0,1\} \qquad (\xi,\eta\in\mathbb{R}).
\end{equation*}
Hence it follows that
\begin{equation}\label{upper bound on nu_p(s)}
	\nu_p(\mathfrak{s}) \leqslant \log_p (2\mathfrak{a}+2\mathfrak{k}) = \frac{\ln(2\mathfrak{a}+2\mathfrak{k})}{\ln p}
\end{equation}
for every odd prime $p$.
From \eqref{upper bound on s} and \eqref{upper bound on nu_p(s)} it follows that
\begin{equation}\label{upper bound on ln(s)}
	\ln\mathfrak{s}\leqslant (\pi(2\mathfrak{k}+1)-1) \ln(2\mathfrak{a}+2\mathfrak{k}),
\end{equation}
where $\pi(\mathfrak{n})$ denotes the number of primes at most $\mathfrak{n}$.

On the other hand, we have
\begin{equation*}
	\mathfrak{s}=\frac{(2\mathfrak{a}+2\mathfrak{k})!\mathfrak{k}!\mathfrak{a}!}{(\mathfrak{a}+\mathfrak{k})!(2\mathfrak{k})!(2\mathfrak{a})!} \cdot \frac{\mathfrak{b}!}{(\mathfrak{b}+\lfloor\vartheta\mathfrak{k}\rfloor)!}.
\end{equation*}
Using Stirling's formula
\begin{equation*}
	\ln (\mathfrak{n}!) = \left(\mathfrak{n}+\frac{1}{2}\right)\ln\mathfrak{n}-\mathfrak{n}+\frac{\ln2\pi}{2}+ r_{\mathfrak{n}},
\end{equation*}
where
\begin{equation*}
	0<r_{\mathfrak{n}}<\frac{1}{12\mathfrak{n}},
\end{equation*}
we obtain
\begin{align}
	\ln\mathfrak{s}&> (\mathfrak{a}+\mathfrak{k})\ln(\mathfrak{a}+\mathfrak{k}) - \mathfrak{k}\ln\mathfrak{k} - \mathfrak{a}\ln\mathfrak{a} +\vartheta\mathfrak{k} -2 \label{expression of ln(s)} \\
	& \qquad\qquad  + \!\left(\mathfrak{b}+\frac{1}{2}\right) \ln\mathfrak{b} - \!\left(\mathfrak{b}+\vartheta\mathfrak{k}+\frac{1}{2}\right)\ln(\mathfrak{b}+\vartheta\mathfrak{k}). \notag\end{align}

Let
\begin{equation*}
	\tilde{\mathfrak{a}}=\frac{\mathfrak{a}}{\mathfrak{k}}, \qquad \tilde{\mathfrak{b}}=\frac{\mathfrak{b}}{\mathfrak{k}}.
\end{equation*}
Note that
\begin{equation}\label{inequalities for new parameters}
	\tilde{\mathfrak{a}},\tilde{\mathfrak{b}}> 1,
\end{equation}
in view of \eqref{a is large} and \eqref{b is large}.
With this notation, we have
\begin{align}
	\ln\mathfrak{s} &> \mathfrak{k}\ln\tilde{\mathfrak{a}} +(\tilde{\mathfrak{a}}+1)\mathfrak{k} \ln\!\left(1+\frac{1}{\tilde{\mathfrak{a}}}\right) -\vartheta\mathfrak{k}\ln\mathfrak{k} +\vartheta\mathfrak{k} - 2 \label{lower bound on ln(s)} \\
	& \qquad\qquad -\vartheta\mathfrak{k}\ln\tilde{\mathfrak{b}} -\!\left((\tilde{\mathfrak{b}}+\vartheta)\mathfrak{k}+\frac{1}{2}\right)\ln\!\left(1+\frac{\vartheta}{\tilde{\mathfrak{b}}}\right) \notag \\
	&> \mathfrak{k}\ln\tilde{\mathfrak{a}} -\vartheta\mathfrak{k}\ln\mathfrak{k} -2 -\vartheta\mathfrak{k}\ln\tilde{\mathfrak{b}} -\!\left(\vartheta\mathfrak{k}+\frac{1}{2}\right)\ln\!\left(1+\frac{\vartheta}{\tilde{\mathfrak{b}}}\right) \notag \\
	&> (1-\vartheta\delta)\mathfrak{k}\ln\tilde{\mathfrak{a}} -\vartheta\mathfrak{k}\ln c -\vartheta\delta\mathfrak{k}\ln\mathfrak{k} -2 - \!\left(\vartheta\mathfrak{k}+\frac{1}{2}\right) \ln (1+\vartheta), \notag
\end{align}
where the first inequality is a restatement of \eqref{expression of ln(s)}, the second follows from
\begin{equation*}
	0<\ln(1+\xi)<\xi \qquad (\xi>0),
\end{equation*}
and the last one follows from \eqref{b is bounded in terms of a} and \eqref{inequalities for new parameters}.

Concerning the prime-counting function $\pi(\mathfrak{n})$, it is known that \cite[(3.6)]{RS1962IJM}
\begin{equation*}
	\pi(\mathfrak{n})<1.25506\,\frac{\mathfrak{n}}{\ln\mathfrak{n}} \qquad (\mathfrak{n}>1).
\end{equation*}
By this, \eqref{upper bound on ln(s)}, and \eqref{inequalities for new parameters}, we have
\begin{align}
	\ln\mathfrak{s} &\leqslant (\pi(2\mathfrak{k}+1)-1) \!\left(\ln\tilde{\mathfrak{a}}+\ln 2\mathfrak{k}\!\left(1+\frac{1}{\tilde{\mathfrak{a}}}\right)\!\right) \label{modified upper bound on ln(s)} \\
	&< \left(1.25506\,\frac{2\mathfrak{k}+1}{\ln(2\mathfrak{k}+1)}-1\right) (\ln\tilde{\mathfrak{a}}+\ln 4\mathfrak{k}). \notag
\end{align}

Combining \eqref{lower bound on ln(s)} and \eqref{modified upper bound on ln(s)}, it follows that
\begin{align}
	\bigg((1-\vartheta\delta)\mathfrak{k} & - 1.25506 \left.\frac{2\mathfrak{k}+1}{\ln(2\mathfrak{k}+1)}+1\right) \ln\tilde{\mathfrak{a}} \label{upper bound on ln(a)} \\
	&< \vartheta\mathfrak{k}\ln c +\vartheta\delta\mathfrak{k}\ln\mathfrak{k} +2 +\!\left(\vartheta\mathfrak{k}+\frac{1}{2}\right) \ln (1+\vartheta) \notag \\
	& \qquad\qquad +\!\left(1.25506\,\frac{2\mathfrak{k}+1}{\ln(2\mathfrak{k}+1)}-1\right) \ln 4\mathfrak{k}. \notag
\end{align}
If we set
\begin{equation*}
	\mathfrak{k}_0=\mathfrak{k}_0(\vartheta,\delta)=\frac{1}{2}\!\left(\exp\!\left(\frac{2.51012}{1-\vartheta\delta}\right)-1\right)>0
\end{equation*}
for example, then we have
\begin{equation*}
	(1-\vartheta\delta)\mathfrak{k}-1.25506\,\frac{2\mathfrak{k}+1}{\ln(2\mathfrak{k}+1)}+1\geqslant \frac{1+\vartheta\delta}{2}>0 \qquad (\mathfrak{k}\geqslant\mathfrak{k}_0).
\end{equation*}
Hence, whenever $\mathfrak{k}\geqslant\mathfrak{k}_0$, it follows from \eqref{upper bound on ln(a)} that $\ln\mathfrak{a}=\ln\tilde{\mathfrak{a}}+\ln\mathfrak{k}$ is bounded in terms of $\vartheta,\delta,c$, and $\mathfrak{k}$, from which and \eqref{b is bounded in terms of a} it follows that there are only finitely many choices for the pairs $(\mathfrak{a},\mathfrak{b})$.

This completes the proof of Proposition \ref{variation of Schur's theorem}.


\begin{thebibliography}{99}

\bibitem{Bannai1977QJMO}
E. Bannai,
On tight designs,
Quart. J. Math. Oxford Ser. (2) 28 (1977) 433--448.

\bibitem{BB2010CM}
E. Bannai and E. Bannai,
Euclidean designs and coherent configurations,
in: R. A. Brualdi et al.~(Eds.),
Combinatorics and graphs,
Contemporary Mathematics, vol.~531,
American Mathematical Society, Providence, RI, 2010, pp.~59--93;
\arxiv{0905.2143}.

\bibitem{BB2012JAMC}
E. Bannai and E. Bannai,
Remarks on the concepts of $t$-designs,
J. Appl. Math. Comput. 40 (2012) 195--207.

\bibitem{BB2014MJCNT}
E. Bannai and E. Bannai,
Tight $t$-designs on two concentric spheres,
Mosc. J. Comb. Number Theory 4 (2014) 52--77.

\bibitem{BBB2014DM}
E. Bannai, E. Bannai, and H. Bannai,
On the existence of tight relative $2$-designs on binary Hamming association schemes,
Discrete Math. 314 (2014) 17--37;
\arxiv{1304.5760}.

\bibitem{BBST2015EJC}
E. Bannai, E. Bannai, S. Suda, and H. Tanaka,
On relative $t$-designs in polynomial association schemes,
Electron. J. Combin. 22 (2015) \#P4.47;
\arxiv{1303.7163}.

\bibitem{BBTZ2017GC}
E. Bannai, E. Bannai, H. Tanaka, and Y. Zhu,
Design theory from the viewpoint of algebraic combinatorics,
Graphs Combin. 33 (2017) 1--41.

\bibitem{BBZ2015PSIM}
E. Bannai, E. Bannai, and Y. Zhu,
A survey on tight Euclidean $t$-designs and tight relative $t$-designs in certain association schemes,
Proc. Steklov Inst. Math. 288 (2015) 189--202.

\bibitem{BBZ2017DCC}
E. Bannai, E. Bannai, and Y. Zhu,
Relative $t$-designs in binary Hamming association scheme $H(n,2)$,
Des. Codes Cryptogr. 84 (2017) 23--53;
\arxiv{1512.01726}.

\bibitem{BI1984B}
E. Bannai and T. Ito,
Algebraic combinatorics I: Association schemes,
Benjamin/Cummings Publishing Company, Menlo Park, CA, 1984.

\bibitem{BZ2019EJC}
E. Bannai and Y. Zhu,
Tight $t$-designs on one shell of Johnson association schemes,
European J. Combin. 80 (2019) 23--36.

\bibitem{BCN1989B}
A. E. Brouwer, A. M. Cohen, and A. Neumaier,
Distance-regular graphs,
Springer-Verlag, Berlin, 1989.

\bibitem{CD2007H}
C. J. Colbourn and J. H. Dinitz (Eds.),
Handbook of combinatorial designs,
Second edition,
Chapman \& Hall/CRC, Boca Raton, FL, 2007.

\bibitem{CR1990B}
C. W. Curtis and I. Reiner,
Methods of representation theory, Vol.~I,
John Wiley \& Sons, New York, 1990.

\bibitem{DKT2016EJC}
E. R. van Dam, J. H. Koolen, and H. Tanaka,
Distance-regular graphs,
Electron. J. Combin. (2016) \#DS22;
\arxiv{1410.6294}.

\bibitem{Delsarte1973PRRS}
P. Delsarte,
An algebraic approach to the association schemes of coding theory,
Philips Res. Rep. Suppl. No. 10 (1973).

\bibitem{Delsarte1976JCTA}
P. Delsarte,
Association schemes and $t$-designs in regular semilattices,
J. Combin. Theory Ser. A 20 (1976) 230--243.

\bibitem{Delsarte1977PRR}
P. Delsarte,
Pairs of vectors in the space of an association scheme,
Philips Res. Rep. 32 (1977) 373--411.

\bibitem{Delsarte1978SIAM}
P. Delsarte,
Hahn polynomials, discrete harmonics, and $t$-designs,
SIAM J. Appl. Math. 34 (1978) 157--166.

\bibitem{DGS1977GD}
P. Delsarte, J. M. Goethals, and J. J. Seidel,
Spherical codes and designs,
Geometriae Dedicata 6 (1977) 363--388.

\bibitem{DS1989LAA}
P. Delsarte and J. J. Seidel,
Fisher type inequalities for Euclidean $t$-designs,
Linear Algebra Appl. 114/115 (1989) 213--230.

\bibitem{DS-G2013JAC}
P. Dukes and J.  Short-Gershman,
Nonexistence results for tight block designs,
J. Algebraic Combin. 38 (2013) 103--119;
\arxiv{1110.3463}.

\bibitem{GVW2019pre}
A. L. Gavrilyuk, J. Vidali, and J. S. Williford,
On few-class $Q$-polynomial association schemes: feasible parameters and nonexistence results,
preprint (2019);
\arxiv{1908.10081}.

\bibitem{Gijswijt2005D}
D. Gijswijt,
Matrix algebras and semidefinite programming techniques for codes,
thesis, Universiteit van Amsterdam, 2005;
\arxiv{1007.0906}.

\bibitem{GST2006JCTA}
D. Gijswijt, A. Schrijver, and H. Tanaka,
New upper bounds for nonbinary codes based on the Terwilliger algebra and semidefinite programming,
J. Combin. Theory Ser. A 113 (2006) 1719--1731.

\bibitem{Go2002EJC}
J. T. Go,
The Terwilliger algebra of the hypercube,
European J. Combin. 23 (2002) 399--429.

\bibitem{Godsil1993B}
C. D. Godsil,
Algebraic combinatorics,
Chapman \& Hall, New York, 1993.

\bibitem{Kageyama1991AC}
S. Kageyama,
A property of $T$-wise balanced designs,
Ars Combin. 31 (1991) 237--238.

\bibitem{Kodalen2019D}
B. G. Kodalen,
Cometric association schemes,
thesis, Worcester Polytechnic Institute, 2019;
\arxiv{1905.06959}.

\bibitem{KLS2010B}
R. Koekoek, P. A. Lesky, and R. F. Swarttouw,
Hypergeometric orthogonal polynomials and their $q$-analogues,
Springer-Verlag, Berlin, 2010.

\bibitem{KS1998R}
R. Koekoek and R. F. Swarttouw,
The Askey scheme of hypergeometric orthogonal polynomials and its $q$-analog, report 98-17, Delft University of Technology, 1998.
Available at \url{http://aw.twi.tudelft.nl/~koekoek/askey.html}

\bibitem{LBB2014GC}
Z. Li, E. Bannai, and E. Bannai,
Tight relative $2$- and $4$-designs on binary Hamming association schemes,
Graphs Combin. 30 (2014) 203--227.

\bibitem{MT2009EJC}
W. J. Martin and H. Tanaka,
Commutative association schemes,
European J. Combin. 30 (2009) 1497--1525;
\arxiv{0811.2475}.

\bibitem{Munemasa1986GC}
A. Munemasa,
An analogue of $t$-designs in the association schemes of alternating bilinear forms,
Graphs Combin. 2 (1986) 259--267.

\bibitem{NS1988IM}
A. Neumaier and J. J. Seidel,
Discrete measures for spherical designs, eutactic stars and lattices,
Nederl. Akad. Wetensch. Indag. Math. 50 (1988) 321--334.

\bibitem{Peterson1977OJM}
C. Peterson,
On tight $6$-designs,
Osaka J. Math. 14 (1977) 417--435.

\bibitem{RS1962IJM}
J. B. Rosser and L. Schoenfeld,
Approximate formulas for some functions of prime numbers,
Illinois J. Math. 6 (1962) 64--94.

\bibitem{SHK2019B}
M. Sawa, M. Hirao, and S. Kageyama,
Euclidean design theory,
Springer, Singapore, 2019.

\bibitem{Schrijver2005IEEE}
A. Schrijver,
New code upper bounds from the Terwilliger algebra and semidefinite programming,
IEEE Trans. Inform. Theory 51 (2005) 2859--2866.

\bibitem{Schur1929SPAW}
I. Schur,
Einige S\"{a}tze \"{u}ber Primzahlen mit Anwendungen auf Irreduzibilit\"{a}tsfragen II,
Sitzungsberichte der Preussischen Akademie der Wissenschaften, 1929
(Gesammelte Abhandlungen, Vol.~III, pp.~152--173).

\bibitem{Stanton1986GC}
D. Stanton,
$t$-designs in classical association schemes,
Graphs Combin. 2 (1986) 283--286.

\bibitem{Szego1975B}
G. Szeg\H{o},
Orthogonal polynomials, Fourth edition,
American Mathematical Society, Providence, RI, 1975.

\bibitem{TFIL2019EJC}
Y.-Y. Tan, Y.-Z. Fan, T. Ito, and X. Liang,
The Terwilliger algebra of the Johnson scheme $J(N,D)$ revisited from the viewpoint of group representations,
European J. Combin. 80 (2019) 157--171.

\bibitem{Tanaka2009EJC}
H. Tanaka,
New proofs of the Assmus--Mattson theorem based on the Terwilliger algebra,
European J. Combin. 30 (2009) 736--746;
\arxiv{math/0612740}.

\bibitem{Tanaka2009LAAa}
H. Tanaka,
A note on the span of Hadamard products of vectors,
Linear Algebra Appl. 430 (2009) 865--867;
\arxiv{0806.2075}.

\bibitem{Terwilliger1992JAC}
P. Terwilliger,
The subconstituent algebra of an association scheme I,
J. Algebraic Combin. 1 (1992) 363--388.

\bibitem{Terwilliger1993JACa}
P. Terwilliger,
The subconstituent algebra of an association scheme II,
J. Algebraic Combin. 2 (1993) 73--103.

\bibitem{Terwilliger1993JACb}
P. Terwilliger,
The subconstituent algebra of an association scheme III,
J. Algebraic Combin. 2 (1993) 177--210.

\bibitem{Vallentin2009LAA}
F. Vallentin,
Symmetry in semidefinite programs,
Linear Algebra Appl. 430 (2009) 360--369;
\arxiv{0706.4233}.

\bibitem{Woodall1970PLMS}
D. R. Woodall,
Square $\lambda$-linked designs,
Proc. London Math. Soc. 20 (1970) 669--687.

\bibitem{Xiang2012JCTA}
Z. Xiang,
A Fisher type inequality for weighted regular $t$-wise balanced designs,
J. Combin. Theory Ser. A 119 (2012) 1523--1527. 

\bibitem{Xiang2018JAC}
Z. Xiang,
Nonexistence of nontrivial tight $8$-designs,
J. Algebraic Combin. 47 (2018) 301--318.

\bibitem{YHG2015DM}
H. Yue, B. Hou, and S. Gao,
Note on the tight relative $2$-designs on $H(n,2)$,
Discrete Math. 338 (2015) 196--208.

\bibitem{ZBB2016DM}
Y. Zhu, E. Bannai, and E. Bannai,
Tight relative $2$-designs on two shells in Johnson association schemes,
Discrete Math. 339 (2016) 957--973.

\end{thebibliography}
\end{document}